\documentclass{IEEEtran}
\usepackage{pdflscape}
\usepackage{verbatim}
\usepackage[cmex10]{amsmath}
\usepackage{amsfonts}
\usepackage{cite}
\usepackage{booktabs}
\usepackage{soul}
\usepackage[ruled,vlined,linesnumbered]{algorithm2e}

\SetKwFor{ParallelForEach}{for each}{in parallel do}{endfor}

\newcommand{\R}{\mathbb{R}}

\usepackage{epstopdf}
 
\usepackage{amsthm} 
\theoremstyle{plain}
\newtheorem{theorem}{Theorem}[section]

\newtheorem{lemma}[theorem]{Lemma}

\newtheorem{assumption}[theorem]{Assumption}

\theoremstyle{definition}

\theoremstyle{remark}
\newtheorem{remark}{Remark}

\usepackage[colorinlistoftodos,bordercolor=orange,backgroundcolor=orange!20,linecolor=orange,textsize=scriptsize]{todonotes}

\usepackage{dcolumn}

\newcolumntype{d}[1]{D{.}{\cdot}{#1} }

\newcommand{\inTR}[1]{}

\SetCommentSty{\small\sf}

\newcommand{\tr}{\textrm{tr}}

   \newcommand{\mfR}{\Re} 
  \newcommand{\mfI}{\Im}

\DeclareGraphicsExtensions{.pdf,.jpeg,.png}
\usepackage{array,color}
\usepackage{url}

\newcommand{\cL}{{\cal{L}}}
\newcommand{\cN}{{\cal N}}

\newcommand{\cG}{{\cal G}}

\newcommand{\cE}{{\cal E}}

\newcommand{\cM}{{\cal M}}

\newcommand{\cY}{{\cal Y}}

\newcommand{\bu}{{\bf u}}

\begin{document}

\title{Optimal Power Flow Pursuit \\ in the Alternating Current Model}

\author{ Jie Liu, 
Antonio Bellon, 
Andrea Simonetto, 
Martin Tak\'a\v{c}, 
Jakub Mare\v{c}ek 
\thanks{
J. Liu is 
at InnoPeak Technology, Inc. (a.k.a. OPPO U.S. Research Center) in Palo Alto, CA, USA.
J. Marecek and A. Bellon are at Czech Technical University, Prague, the Czech Republic. 
M. Tak\'a\v{c} is at the Mohamed Bin Zayed University of Artificial Intelligence (MBZUAI) in Abu Dhabi, United Arab Emirates. 
A. Simonetto is with ENSTA Paris, Institut Polytechnique de Paris, France. 
This work of J. Liu and M. Tak\'a\v{c} has been supported by IBM Corporation and  
the National Science Foundation under grants no. NSF:CCF:161871 and NSF:CMMI-1663256.
J. Marecek and A. Bellon acknowledge support of the OP RDE funded project CZ.02.1.01/0.0/0.0/16\_019/0000765 ``Research Center for Informatics''.
}
}

\maketitle

\begin{abstract}
Transmission-constrained problems in power systems can be cast as polynomial optimization problems whose coefficients vary over time. 
We consider the complications therein and suggest several approaches.
We illustrate one of the approaches in detail in the case of alternating-current optimal power flow (ACOPF) problems.
For the time-varying ACOPF, we provide an upper bound for the difference between the optimal cost for a relaxation using the most recent data and the current approximate optimal cost generated by our algorithm. This bound is a function of the properties of the instance and the rate of change of the coefficients over time.
Moreover, we bound the number of floating-point operations to perform between two subsequent updates to ensure a bounded error.
\end{abstract}

\begin{IEEEkeywords}
Numerical analysis (Mathematical programming), optimization, Power system analysis computing
\end{IEEEkeywords}


\section{Introduction} 

Renewable energy sources (RESs) have created several new challenges for power system control and analysis. 
In particular, power quality and reliability can be undermined when RESs are used widely and when all the available power is injected. 
Furthermore, in distribution systems, overvoltages might become frequent. 
There, as well as in transmission systems, rapid changes in power output can even cause power flow reversals, as well as unexpected losses and transients that current systems cannot handle.  
Therefore, real-time control mechanisms must be designed, for example, to limit real power at RESs inverters, taking into account transmission constraints. 
As a main complication, the transmission-constrained problems coming from the alternating-current model are non-convex and non-linear. 
In both theory \cite{Baker2021} and practice, linearizations tend to produce infeasible solutions.
Approaches applying Newton method~\cite{Idema2012} to the non-convex problem in a rolling-horizon framework often perform well in practice, as long as the changes are limited. However, in general, they provide little to none theoretical guarantees on their performance.
In contrast, solutions to specific relaxations (see, for example, ~\cite{lavaei2012zero, Ghaddar2015}) coincide under mild assumptions with those to non-convex problems, for all initial points. 
As a main complication, it can actually take a long time to solve the relaxation, so that meanwhile the inputs can change significantly, so that when the solution is available, it might already be outdated. 
This requires providing time-varying solutions for time-varying optimization problems (see~\cite{SPM, Simonetto2020} for recent surveys and results). 
In this article, we propose a coordinate-descent algorithm \cite{Marecek2017,Liu2017}, where each step has a closed-form solution.
This reduces the computational burden per iteration to a very limited number of floating-point operations.   
In the case of alternating-current optimal power flows (ACOPFs),
we provide an upper bound on the difference between the current
optimum $\cL^{k,*}$ of the relaxation derived using the most recent update in expectation and the approximate current cost $\cL^k$, that is,
$\lim\sup_{k \to \infty} \mathbb{E}[\cL^k - \cL^{k,*}]$.
This bound is a function of the properties of the instance
and the extent of updates to the instance.

This provides a novel perspective on time-varying optimization in power systems in two ways.
First, we do not consider linearization, with feedback or not,~\cite{Bolognani2015, 7480375, Bolognani2016, Hauswirth2017, ORTMANN2020, PICALLO2020}, but just consider the non-linear non-convex problem.
Secondly, we can analyze the properties of this approach in some detail: 
under a variant of the Polyak\,--\,{\L}ojasiewicz condition, we show a bound on the error in tracking the trajectory of optimal solutions.
Furthermore, we show that the delay in applying the update is $O(np)$ for $n$ nodes, each connected to at most $p$ other nodes, thanks to the closed-form solution for every coordinate-wise step.
As we illustrate in computational demonstrations on the IEEE 37-node and IEEE 118-node test systems, tracking of ACOPF solutions is practically possible. 
\begin{remark}
A preliminary version of this paper has appeared as~\cite{8442544}. 
Subsequently, we have developed further insights \cite{bellon2021time} into the problem, as well as a path-following procedure to track the trajectory of solutions to time-varying semidefinite programming \cite{Bellon2022}.
The present paper extends beyond this preliminary version \cite{8442544} in several ways.
First, we present the original algorithm in a broad framework that highlights the complications encountered in multiple algorithmic approaches to the problem
in Sections \ref{sec:perspective} and \ref{sec:defns}.
Second, our assumptions are discussed in more detail and proofs of our results are given. 
Finally,
our computational experiments are extended towards the use of the IEEE 118-bus test system.
\end{remark}

\section{The Problem}


Consistent with state-of-the-art literature \cite{lavaei2012zero,molzahn2011,Ghaddar2015,Marecek2017}, 
we consider the  model of a power system representing it by a graph 
with nodes $\cN := \{1,\ldots,N\}$ which are connected by some edges $\cE := \{(m,n)\} \subset \cN  \times \cN$.
A particular subset of nodes, denoted with $\cG \subseteq \cN$, contains a number $N_{\cG}:= |\cG|$ of controllable generators.
We further assume that the model is two-terminal and  pi-equivalent
Here, time is discretized and varying on a set $\{k \tau\}_{k\in\mathbb{N}}$, where  $k$ is the multiplier and $\tau > 0$ is the period, chosen to capture the variations in loads, as well as in ambient conditions. 
In our model, we use the following variables:
\begin{itemize}
\item $V_n^k \in \mathbb{C}$, denoting the phasors for the line-to-ground voltage at the $k$-th time period
\item $I_n^k \in \mathbb{C}$, denoting the current injected at node $n$ over the $k$-th time period
\item $P_{n}^k$ and $Q_n^k$, denoting the active and reactive power injected at $n \in \cG$ at the $k$-th period of time
\end{itemize}
These variables are stacked into $N$-dimensional complex vectors $V^k := (V_1^k, \ldots, V_N^k)^T \in \mathbb{C}^{N}$ 
and $I^k := (I_1^k, \ldots, I_N^k)^T \in \mathbb{C}^{N}$. 
Combining Kirchhoff's and Ohm's circuit laws, one can derive the linear equations:
$
I^k
= 
y
V^k
$,
where $y \in \mathbb{C}^{N \times N}$ is the admittance matrix of the system.
We then add a node $0$ to $\cN$ and fix for it the voltage magnitude $\rho_0$ and angle $\theta_0$, so that at any time $k$ we have $V_0^k = \rho_0 e^{\mathrm{j} \theta_0}$.
We then assume a constant load at each node $n \in \cN \setminus \cG$ and each time $k$, defining the quantities 
$P_{\ell,n}^k$ and $Q_{\ell,n}^k$ as the real and reactive demands. 
Moreover, we consider a set of nodes $\cM \subseteq \cN$ where voltage regulation is possible, for which
$V^{\mathrm{min}}$ and $V^{\mathrm{max}}$ are the relative voltage limits.
For a given generator $n \in \cG$, $P_{\textrm{av},n}^{k}$ denotes the maximum active power generation at time $k$.
In a photovoltaic system for example, $P_{\textrm{av},n}^{k}$ is a function of irradiance, which is bounded from above by a limit on the inverter. 
Finally, $S_n$ is the rated apparent power.

Typically, an off-line optimization problem, known as the alternating-current optimal power flow (ACOPF), is considered.
At a given time $k \tau$ 
this can be put in the simple form
\begin{subequations} 
\label{Pmg}
\begin{align}
\min_{V, I,  \{P_i, Q_i \}_{i \in \cG} } \,\, h^k(\{V_i\}_{i \in \cN}) + \sum_{i \in \cG} f_i^k(P_i, Q_i)
\label{mg-cost}
\end{align}
\vspace{-1cm}
\begin{align} 
&\text{s.t.} &&I^k  = y V^k, &&
\label{eq:iYv}
\\ 
&&&V_i I_i^* = P_i - P_{\ell,i}^k + \mathrm{j} (Q_i - Q_{\ell,i}^k), && i \in \cG,   \label{mg-balance-I}
\\
&&&V_n I_n^*  = - P_{\ell,n}^k - \mathrm{j} Q_{\ell,n}^k, && n \in \cN \backslash \cG, \label{mg-balance-L}
\\
&&&V^{\mathrm{min}}  \leq |V_i| \leq V^{\mathrm{max}},  &&  i \in \cM,
\label{mg-Vlimits} 
\\
&&&0  \leq {P}_{n}  \leq  \min \{ P_{\textrm{av},n}^k, S_{n} \}, && n \in \cN,
\\
&&&{Q}_{n}  \leq  S_{n}, && n \in \cG,
\label{mg-PV} 
\end{align}
\end{subequations} 
where $h^k(\{V_i\}_{i \in \cN})$ captures system-level objectives and $f_i^k(P_i, Q_i)$ is a time-varying function that specifies performance objectives for generator $i$.

This simple form of the ACOPF problem can be lifted in a higher dimension \cite{Marecek2017}.
In order to simplify the notation, when not needed, we omit the time index $k$.
We consider the following $2N\times 2N$ real matrices:
\begin{align}
M_i &:=  
\begin{bmatrix}
e_i e_i^T & 0 \\
0 & e_i e_i^T 
\end{bmatrix},
\label{defMk}
\\
y_i &:= e_i e_i^T y,
\\
Y_i &:= \frac12 
\begin{bmatrix}
\mfR(y_i + y_i^T) & \mfI(y_i^T - y_i)\\
\mfI(y_i - y_i^T) & \mfR(y_i + y_i^T) 
\end{bmatrix},
\label{defYk}
\\ 
\bar{Y}_i &:= -\frac12 
\begin{bmatrix}
\mfI(y_i + y_i^T) & \mfR(y_i - y_i^T)\\
\mfR(y_i^T - y_i) & \mfI(y_i + y_i^T) 
\end{bmatrix},
\label{defbarYk}  
\end{align}
where $e_i$ is the $i$-th vector of standard basis of $\mathbb{R}^{N}$.
We can then introduce the following new variables:
\begin{align}
x   &:=  \begin{bmatrix}\mfR{V}\\ \mfI{V}\end{bmatrix}, &&
\\
t_i &:=  \tr(Y_i xx^T), &&  i\in \mathcal{N},
\label{deft} \\ 
g_i &:=  \tr(\bar Y_i xx^T), &&  i\in \mathcal{N},
\label{defg} \\ 
h_i &:=  \tr(M_i xx^T), &&  i\in \mathcal{N}. \label{defh}
\end{align} %
Using the variables $t_i, g_i, z_i$ for $i\in \mathcal{G}$, $h_i$ for $i\in\mathcal{N}$, and $x$, we can reformulate the problem as follows:
\begin{subequations}
\label{lifted}
\begin{align}
 \hspace{-0.1cm}\min_{x\in\mathbb{R}^{2N}} &\sum_{i\in\mathcal{G}} \hspace{-0.05cm} c_i [P_{l,i} +\tr(Y_i xx^T)]^2 \hspace{-0.05cm}+\hspace{-0.05cm} d_i [Q_{l,i}  
+\tr(\bar Y_i xx^T)]^2 
\label{lifted1} 
\end{align}
\vspace{-0.5cm}
\begin{align}
&\text{s.t.} &&
t_i = \tr(Y_i xx^T), &&  i\in \mathcal{N} ,
\label{lifted2} 
\\
&&&g_i = \tr(\bar Y_i xx^T), &&  i\in \mathcal{N} ,
\label{lifted3} 
\\
&&&h_i = \tr(M_i xx^T), &&  i\in \mathcal{N} ,
\label{lifted4} 
\\
&&&V_{min}^2\leq h_i \leq V_{max}^2, &&  i\in\mathcal{N},
\\
&&&z_i = (P_{l,i} + t_i)^2 + (Q_{l,i} + g_i)^2, &&  i\in\mathcal{G},
\label{lifted6}
\\
&&&z_i \leq S_i^2, &&  i\in\mathcal{G},
\\
&&&-P_{l,i} \leq  t_i \leq P_{pv}- P_{l,i}, &&  i\in\mathcal{G},
\\
&&&t_i = -P_{l,i}, &&  i\in\mathcal{N}\backslash\mathcal{G}, 
\label{lifted9}
\\
&&&g_i =-Q_{l,i}, &&  i\in \mathcal{N}\backslash\mathcal{G}.
\label{lifted10}
\end{align}
\end{subequations} 
The problem could be further extended \cite{molzahn2011,Marecek2017} to examine phase shift and tap change transformers in single-line thermal limits. Being outside the scope of this paper, we do not explore such extensions.

\subsection{A Perspective on the Problem}
\label{sec:perspective}


\begin{table*}[th!]
    \centering
    \begin{tabular}{l|l}
    \midrule[1pt]
        \textbf{Problem type}&\textbf{Behaviors (examples)}
         \\
         \midrule[1pt]
         Time-varying polynomial optimization \eqref{eq:tv_poly_opt}  & (a.) Loss of smoothness\\
         & (b.) Loss of continuity\\
         & (c.) Continuous change of dimension\\
          \midrule 
         TV-SDP, assumed to be exact for the TV-POP  & (a.) Loss of smoothness\\
         & (b.) Loss of continuity\\
         & (c.) Continuous change of dimension\\
         \midrule[1pt] \textbf{Problem type}&\textbf{Behaviors (complete classification)}\\
          \midrule[1pt]
         TV-SDP with LICQ,   continuous data, Slater's condition, and a \textit{non-singular}
         point
         & (a.) Loss of smoothness\\
         & (b.) Loss of continuity\\
          \midrule 
         TV-SDP with LICQ,  continuous data, Slater's condition, without a \textit{non-singular} points & (a.) Loss of smoothness\\
         & (b.) Loss of continuity\\
         & (c.) Continuous change of dimension\\
         & (d.) Accumulation point for a set of irregular points\\
          \midrule 
         TV-SDP with LICQ,  continuous data, without Slater's condition  
         & (a.) Loss of smoothness\\
         & (b.) Loss of continuity\\
         & (c.) Continuous change of dimension\\
         & (d.) Accumulation point for a set of irregular points\\
         & (e.) Unattained optima\\ & (f.) Positive duality gap\\
        \midrule[1pt]
    \end{tabular}
    \caption{Assumptions on the time-varying problem ordered from the most restrictive (top) and the associated behaviors. See Section \ref{sec:defns} for definitions and discussion. }
    \label{tab:tvsdp_behaviors}
\end{table*}

A useful perspective on the \emph{non-convex} optimization problem \eqref{Pmg} is to see it as a special case of a \textit{polynomial optimization problem} (POP) of the form
\begin{equation}
\label{eq:poly_opt}
    \begin{aligned}
        &\min_{x\in\mathbb{R}^N} && p_0(x) && \\
        &\text{\ \ s.t.} && p_i(x)\ge 0, && i\in \mathcal{I},
    \end{aligned}
\end{equation}
where $p_i \in\mathbb{R}[x]$ for all $i\in\{0\}\cup\mathcal{I}$ are real polynomials in $N$ variables. One of the most reliable and successful tools for solving \eqref{eq:poly_opt} is the moment-sum-of-squares hierarchy \cite{henrion2020moment}. This approach utilizes a nested sequence of \textit{semidefinite optimization problems} (SDP), i.e., linear optimization problems including linear matrix inequalities which constrain a linear matrix combinations of variables to be positive semidefinite. The $j$-th SDP, where $j$ is the \textit{order} of the relaxation, is in dimension $d^j<d^{j+1}$:
\begin{equation}
\label{eq:sdp}
    \begin{aligned}
        & \min_{y} && \sum_{i}c_i y_i\\
        & \text{\ \ s.t.} && A^j_0+\sum_{i}y_iA^j_i\in\mathbb{S}_+^{d^j}.
    \end{aligned}
\end{equation}
Problem \eqref{eq:sdp} is a convex relaxation of problem \eqref{eq:poly_opt}, whose solution gives an upper bound for the optimal value of \eqref{eq:poly_opt}. Increasing the size of the relaxation, the related bound for the optimal value of the polynomial problem can only improve. Furthermore, generically, there always exists a problem of the hierarchy of relaxations whose optimal value coincides with the optimal value of the original problem. In general, it is not possible to know \textit{a priori}, which order of the hierarchy yields such an exact relaxation, but spectral conditions are available to certify that a certain instance allows for the first applicable order of the hierarchy to be used \cite{lavaei2012zero} and to certify that a certain relaxation is exact \cite{molzahn2013sufficient,Ghaddar2015}, allowing to design algorithms that solve a sequence of increasing size of semidefinite problems and stop in a finite number of steps, returning the optimal value of \eqref{eq:poly_opt} and the optimizer when the conditions cited are met.
The convenience of such algorithms is based on the complexity of solving a semidefinite optimization problem, for which an accurate-as-desired solution can be found in a time that is polynomial in the size of the inputs.

\section{Time-Varying Polynomial Optimization and Time-Varying Semidefinite Programming: An Inventory of Behaviors and Approaches}
\label{sec:defns}

However, in constrained time-varying optimization, non-trivial complications may arise. 
These concern the geometry of the trajectory of their solutions and are intrinsic to time-varying semidefinite programming and time-varying polynomial optimization. 
Suppose that for a given sequence of times $T=\{t_0,\dots,t_k,\dots,t_{f}\}$ indexed by $k$, we are interested in solving a \textit{time-varying POP} depending on the time index $k$ of the form
\begin{equation}
\label{eq:tv_poly_opt}
\tag{TV-POP}
    \begin{aligned}
        &\min &&p^k_0(x)&&\\
        &\mathrm{\ \ s.t.} &&p_i^k(x)\ge 0,&& i\in \mathcal{I},\\
        &&&x\in\mathbb{R}^N.&&
    \end{aligned}
\end{equation}
While our understanding of \eqref{eq:tv_poly_opt} is not complete, yet, 
under certain favorable conditions, we may be able to replace \eqref{eq:tv_poly_opt} with a time-varying version of \eqref{eq:sdp}, which we refer to as TV-SDP. This is possible by means of the moment-sum-of-squares hierarchy \cite{henrion2020moment}
(in the particular case of ACOPF in \eqref{Pmg}, 
the conditions are well understood \cite{lavaei2012zero}). 
Fortunately, the behaviors of TV-SDP solution trajectories have been recently characterized \cite{bellon2021time},
and the behaviors of TV-SDP serve as a subset of behaviors of \eqref{eq:tv_poly_opt}.
Table \ref{tab:tvsdp_behaviors} shows that even under quite restrictive hypotheses, bad behavior can occur.

For simplicity, let us consider the continuous-time setting first and let us define the terms used in Table \ref{tab:tvsdp_behaviors}. The Linearly Independent Constraints Qualification (LICQ) assumption requires that the linear constraints defining the feasible region of the problem should be linearly independent. This assumption is almost costless, as one can always eliminate redundant constraints without any loss of generality. It is indeed a standard hypothesis. 
Assuming continuity of data with respect to the time parameter excludes some application cases (e.g., the presence of deadbands), but is a sensible hypothesis in a study framework, like the one of this chapter. It ensures, for simplicity's sake, that an observed irregular behavior does not come from the parametrization of the problem data.
Slater's condition requires the existence of a strictly feasible point, that is, a point in the relative interior of the feasible set \cite{goldfarb1999parametric, ahmadi2021time}. If such a condition does not hold, there are two irregular behaviors that one might observe. The first one (e.) is that the optimum might be unattained. This means that there is a sequence of feasible points with values converging to the optimal value, but there is no feasible point such that its value is the optimal one. In other words, the optimal value is a supremum and not a maximum. The second one (f.) is that the duality gap, the absolute difference between the primal and dual optimal value is strictly positive \cite[Section 2.3]{drusvyatskiy2017many}. This latter phenomenon may prevent one to use primal-dual methods. Because of this, assuming Slater's condition is standard practice, as it avoids such phenomena (see the last two rows in Table \ref{tab:tvsdp_behaviors}). Yet, other type of irregularities are not ruled out by this condition, even when one keeps on assuming LICQ and the continuity of the data.
The first possible issue is (a.) a loss of smoothness of the curve drawn by the solution to TV-SDP. Notice that the cause of this phenomenon is geometric in nature, as a consequence of the feasible region structure; not necessarily because of a loss of smoothness in the time dependence. More dramatic behaviors involve a change in the dimension of the optimal set. In particular, when this happens, two cases are possible: (c.) the change of dimension happens in a continuous fashion, creating continuous bifurcations in the trajectory of solutions; (b.) the change of dimension causes a  discontinuity in the trajectory of solutions \cite{hauenstein2019computing}. Here we are using the Painlevé-Kuratowsky continuity notion, which extend the usual continuity notion to set-valued function.
As an extreme case, one can even observe (d.) points in time whose neighborhood of any duration always comprises a change-of-dimension point. 
Cases (c.) and (d.) can be ruled out by assuming the existence throughout the time parametrization interval of a \textit{non-singular point}, a point at which one can use the implicit function theorem on a subset of the optimality conditions, exploiting the fact that any constructible
set is either a finite set or the complement of a finite set \cite{bellon2021time}.
However, situations (a.) and (b.) can show up even in a best-looking set up.
We refer to a companion paper \cite{bellon2021time} for formal definitions and further details
and to Table~\ref{tab:translation} for a rough translation of these terms to power-engineering language.

Next, let us consider a discrete-time setting. Indeed, we are dealing with a discrete sequence of instances of TV-SDP, in practice, 
where each member of the sequence is an instance corresponding to a fixed time instant. 
The assumption of continuity of the trajectory of the inputs above (and in Table \ref{tab:tvsdp_behaviors}) needs to be ``discretized'', 
by bounding the change from one instance to the next one. This can be done in three ways: 
bounding a norm of the change in inputs, 
bounding a norm of the change changes in the objective function, 
and bounding a norm of the change in the Lagrangian. 
By bounding the norm of changes in the Lagrangian, 
we are making sure that the change in the optimal value must be bounded,
but we allow for the infeasibility introduced by the changing inputs.
This last option seems more natural and  more general than the previous two, as the Lagrangian captures the essence of an optimization problem and makes it possible to measure the tracking performance (see Theorem \ref{theorem.inexact}). 

Let us now consider three approaches to solving 
\textit{time-varying POP} 
\eqref{eq:tv_poly_opt}.

\subsection{Repeated Solving of the Time-Invariant Problem}

A straightforward approach would iteratively solve an off-line optimization problem at a fixed frequency. 
One could discard the information concerning the preceding steps, or not. 
In the former case, this approach may consume more computational resources than necessary. 
In the latter case, this approach is known as warm-starting \cite{Gondzio1998,Yildirim2002,Colombo2011}.
Much of the general-purpose work in this has focused on the use of 
interior-point methods \cite{Gondzio1998,Colombo2011}, 
where a small number of computationally-demanding iterations suffice \cite{Gondzio2012} to reach machine precision. 
While we are not aware of any work on warm-starting SDP solvers, 
one could hypothetically consider a fixed step in the moment-sum-of-squares hierarchy \cite{henrion2020moment,Ghaddar2015} without a warm start.
In most of the real-world applications, this would, however, result in an overly expensive procedure. 
This suggests that one should instead consider a path-following strategy.

\subsection{Path-following for Convex Approximations}

The path-following approach makes use of local information on the current problem instance to predict a solution for the next problem instance after a sufficiently small time step. 
This predicting procedure is often combined with a corrector step exploiting the information of the new problem instance to correct the solution predicted by the predictor step. Procedures known as predictor-corrector are able to merge these two steps in one, where the step to the new approximate solution is found solving a convex quadratic problem \cite{kungurtsev2014sequential}. 
A number of papers, for example \cite{Bolognani2015, 7480375,7842813,7859385,8013070,Bolognani2016, Hauswirth2017,ORTMANN2020,PICALLO2020}, have focussed  
 on linearizations of the OPF problem and path-following procedures therein, possibly employing feedback to correct for model mismatches and linearization errors. 
While we are not aware of any work on path-following for time-varying SDP (TV-SDP), 
one could hypothetically employ path-following to the moment-sum-of-squares hierarchy, possibly at a fixed step therein, again. 
This suggests to apply the Lasserre's hierarchy in a time-varying framework. This, however, might lead to complications that limit the ability to track the solution to time-varying SDP using path-following strategies as developed in \cite{Bellon2022}.
As discussed above in this section, even in a best-case scenario, the trajectory of the solution may lose smoothness or even continuity, making it hard or impossible to use local information.

\subsection{Path-following under Polyak\,--\,\L{}ojasiewicz inequality}
\label{sec:PLapp}

Finally, one could go beyond the convex problems. Within power systems, gradient methods \cite{Elia-Allerton13,XIE2022107859}, Newton method, and L-BFGS \cite{7929408}   
have been applied to the general non-convex problem, without guarantees. 
Guarantees are, however, possible under assumptions on the behavior of the objective function around local optima.
In particular, the so called Polyak\,--\,\L{}ojasiewicz (PL) inqualities \cite{Polyak63,lojasiewicz1963propriete} concern the growth of the gradient around local optima. 
Functions satisfying the PL inequality are neither a subset nor a superset of convex functions \cite[Section 5.1]{bolte2021curiosities}.
Since 1960s \cite{Polyak63,lojasiewicz1963propriete},
a number of variants have been introduced \cite[e.g.]{kurdyka1998gradients,karimi2016linear}.
The so-called Kurdyka\,--\,{\L}ojasiewicz variant \cite{kurdyka1998gradients} is known to be satisfied 
by any semiagebraic set, including \eqref{eq:tv_poly_opt}.
A stricter, but particularly suitable variant is the local proximal PL inequality (cf. Assumption \ref{asmPL} below).
Luo-Tseng error bounds \cite{tseng2010approximation} could also be seen as a variant. 
Below, we show that this approach, even if restrictive, presents rather a solid alternative. 

\begin{table}[]
    \centering
    \begin{tabular}{c|c}
    \midrule[1pt]
    \textbf{CQ} & \textbf{Power engineering} \\ 
    \midrule[1pt]
        LICQ & Pre-processing applied \\
        Continuity & No deadband etc. \\
        Slater's &  Strictly feasible power flow\\
        
        \midrule 
        Invexity & Holds at least for 3 buses \cite{bestuzheva2019invex} \\
        Convexity & In radial networks or without losses \\
        \midrule 
    KL Inequality & Satisfied \cite{kurdyka1998gradients} \\
    PL Inequality (Ass. \ref{asmPL}) & Empirical evidence \cite{bestuzheva2019invex} \\
    \midrule[1pt]
    \end{tabular}
    \caption{Translation of the mathematical properties of the time-varying problems to power-engineering language 
    }
    \label{tab:translation}
\end{table}

\section{One Approach Elaborated: \\
a Randomized Coordinate-Descent Algorithm}
Our approach elaborates upon the path-following under the PL inequality sketched out in Section \ref{sec:PLapp} above.
It is based on first-order methods, namely a randomized coordinate-descent algorithm, applied to the Lagrangian relaxation of \eqref{lifted}. We begin by setting $\xi:= (x, t, h, g, z, \lambda^t, \lambda^g, \lambda^h, \lambda^z)\in\R^d$ for a suitable dimension $d\in\mathbb{N}$, where the $\lambda$s are the Lagrangian multipliers relative to constraints (\ref{lifted}b-c-d-f), and by considering
\begin{align*}
& \cL( \xi, \mu ):=
\\
& \sum_{i\in\mathcal{G}} \left\{ c_i [P_{l,i} +\tr(Y_i xx^T) ]^2 
+ d_i [Q_{l,i}  +\tr(\bar Y_i xx^T) ]^2\right\}
\\
& - \sum_{i\in\mathcal{N}} \lambda_i^t \left[\tr(Y_i xx^T)  - t_i\right]
+ \frac\mu2\sum_{i\in\mathcal{N}} \left[\tr(Y_ixx^T)-t_i\right]^2
\\
& - \sum_{i\in\mathcal{N}} \lambda_i^g \left[\tr(\bar Y_i xx^T)  - g_i \right] 
+ \frac\mu2\sum_{i\in\mathcal{N}} \left[\tr(\bar Y_i xx^T) -g_i\right]^2
\\
& - \sum_{i\in\mathcal{N}} \lambda_i^h \left[\tr(M_i xx^T) - h_i \right] 
+ \frac\mu2\sum_{i\in\mathcal{N}} \left[\tr(M_i xx^T)-h_i \right]^2\\
& - \sum_{i\in\mathcal{G}} \lambda_{i}^z \left[(t_i+P_{l,i})^2 +(g_i+Q_{l,i})^2 - z_{i}\right]   
\\
& + \frac\mu2\sum_{i\in\mathcal{G}} \left[(t_i+P_{l,i})^2  + (g_i+Q_{l,i})^2 - z_{i}\right]^2.
\end{align*} 
This augmented Lagrangian is deeply related with SDP relaxations \cite{lavaei2012zero,Ghaddar2015}, where $xx^T$ is replaced by a matrix $X$, which is then required to satisfy the positive semidefiniteness constraint $X \succeq 0$,
as described in \cite{lavaei2012zero,Marecek2017}.
We optimize $\cL$ on the polyhedral feasible set $\cY\subset\mathbb{R}^d$ defined by the inequalities:
\begin{align}
&V_{min}^2\leq h_i \leq V_{max}^2, &&  i\in\mathcal{N},
\\
&z_i \leq S_i^2, &&  i\in\mathcal{G}.
\\
-&P_{l,i} \leq  t_i \leq P_{pv}- P_{l,i}, &&  i\in\mathcal{G},
\end{align} 
We further define $\chi$ as the indicator function of $\cY$, so that $\chi(\xi)$ is zero if $\xi\in\cY$ and infinity otherwise.

Denoting an initial point as $\xi^0$ and 
the $k$-th iterate as $\xi^k$, we update the $i^k$-th coordinate of the next iterate $\xi^{k+1}_{i^k}$ by
\begin{align}\label{eq:updateopt_compact}
\hspace{-0.28cm}\arg \min_{\alpha \in \mathbb{R}}\hspace{-0.05cm} \left [ \hspace{-0.02cm}\alpha \nabla_{i^k} \cL(\xi^k\hspace{-0.05cm},\mu) \hspace{-0.05cm}+\hspace{-0.05cm} \frac{L }{ 2}\alpha^2 \hspace{-0.05cm}+\hspace{-0.05cm} \chi_{i^k}(\xi_{i^k} \hspace{-0.05cm}+ \alpha) \hspace{-0.05cm}-\hspace{-0.05cm} \chi_{i^k}(\xi_{i^k}) \hspace{-0.05cm}\right ]\hspace{-0.05cm},
\end{align}
where $\nabla_{i} \cL$ is the restriction of the gradient to coordinate $i$ and $\chi_{i}$ is the coordinate-wise indicator function.
This can be seen as coordinate-wise minimization applied to problem
\begin{align}
\label{f+g}
\arg \min_{\xi^k} \cL(\xi^k, \mu) + \chi(\xi^k).  
\end{align}
We adopt the following procedure
\begin{algorithm}[t]
\caption{A randomized coordinate-descent algorithm for ACOPF Pursuit}
 \textbf{Input:} \textrm{data for} \eqref{lifted} \textrm{at each} $\{k\tau\}_{k\in\{0,\dots,K\}
}$, initial point $\xi_0$ \\
 \textbf{Output:} \textrm{sequence of solutions} $\Xi=\{\xi_k\}_{k\in\{0,\dots,K\}
}$\\
 \textbf{initialize} $\Xi=\{\xi_0\}$, \textrm{choose} $\mu\in[0,\bar \mu]$ \;
  \textbf{for} $k = 0$ to $K$ \;
  \quad \textrm{choose} $i^k\in\{1,\dots,d\}$  \textrm{with uniform probability} \;
  \quad \textrm{solve} \eqref{eq:updateopt_compact}, \textrm{obtain} $\xi^{k+1}_{i^k}$ \textrm{and set} $\xi^{k+1}_{j}=\xi_j^{k} \text{ for }j\neq i^k$ \;
  \quad \textrm{append} $\xi^{k+1}$ to $\Xi$ \;
 \textbf{return} $\Xi$ \;
\end{algorithm}
\noindent
As a crucial observation, \eqref{eq:updateopt_compact} admits a closed-form solution.
Indeed, since $\cL$ is a degree-4 polynomial in $x$,
the optimality conditions are cubic, and hence each root of the uni-variate problem has a closed-form.
One can consequently enumerate these finite solutions and choose the one realizing the minimum.
As for the other variables, $\cL$
is at most quadratic on $\cY$.
On the one hand this allows the analysis of the per-iteration complexity, as presented in Section~\ref{sec:periteration}; on the other hand, it implies computational performance that are possibly excellent.

\section{Tracking errors and convergence rates}

We begin by considering the properties of the 
Lagrangian.

\begin{lemma}[Lipschitz continuity of the Lagrangian gradient]
\label{lemma-Phi}
Let $B_r(\xi^*) \subset \mathbb{R}^d$ be a Euclidean ball with center $\xi^*$
and finite radius $r$.
Then 
$\nabla_\xi \cL$ is coordinate-wise Lipschitz continuous on $B_r(\xi^*)$, i.e.,
there exists a constant $L$ such that for every
$\alpha \in \mathbb{R}$,
$\xi \in B_r(\xi^*)$, and for any index $i \in \{1,\dots,d\}$ 
such that $\xi + \alpha e_i \in B_r(\xi^*)$, 
the following upper-bound is satisfied
\begin{align}
\label{lip_coo}
 \cL(\xi + \alpha e_i, \mu) \leq \cL(\xi, \mu) + \alpha \nabla_i \cL(\xi, \mu) + \frac{L }{ 2} \alpha^2,
\end{align}
where $e_i$ is the $i$-th unit vector. 
\end{lemma}
\begin{proof}
For fixed $\mu$, the function $\cL(\xi, \mu)$ is a polynomial function of $\xi$. We can then simply set
\[
L := 
\max_{\substack{i\in\{1,\dots,d\},\\ \xi \in B_r(\xi^*)}}  \left| \frac{\partial^2 \cL(\xi,\mu)}{\partial \xi_i^2}  \right|,
\]
which is well-defined, as the second-order derivatives of $\cL$ are polynomials, admitting maximum on the closed ball $B_r(\xi^*)$.
\end{proof}

Let us now fix a local minimizer $\xi^*$ for $\cL(\xi,\mu)$ for fixed $\mu$. 
Throughout this paper, we assume that $\mu \in [0, \bar \mu]$, for some $\bar \mu$.
Furthermore, based on \cite{Polyak63,lojasiewicz1963propriete,karimi2016linear},
we make an assumption relating the growth of the gradient to sub-optimality.
\begin{assumption}[Local proximal PL inequality]
\label{asmPL}
Given a local minimizer $\xi^*$, 
and a fixed $\mu \in [0,\bar \mu]$, there 
exists a finite radius $r>0$ and a constant $\sigma_\cL > 0$ such that the map $\nabla \cL$ satisfies the local proximal 
PL inequality, i.e., the following inequality holds for every $\xi \in B_r(\xi^*)$ 
\begin{equation}\label{prox-pl}
\frac 1 2\mathcal{D}_\chi(\xi, L) \geq   \sigma_\cL [\cL(\xi,\mu) - \cL(\xi^*,\mu)],
\end{equation} 
where $\chi$ is the indicator function as in \eqref{f+g} and $\mathcal{D}_\chi(\xi, \alpha)$ is defined as
\begin{equation}
-2\alpha \min_{\xi'} \left[ \langle \nabla \cL(\xi,\mu) , \xi' - \xi \rangle + \frac{\alpha}{2}|| \xi' - \xi ||^2 + \chi(\xi') - \chi(\xi) \right].
\nonumber
\end{equation}
\end{assumption}

A global convergence analysis is non-trivial, due to the non-convexity of the Lagrangian, \cite{lavaei2012zero,Ghaddar2015} and requires either additional assumptions \cite{lavaei2012zero} or the employment of SDP relaxations \cite{Ghaddar2015} (for examples of their usage see, e.g., \cite{Marecek2017,boumal2016}).
For the latter case, even if the associated assumptions are in general well-known \cite{Burer2003,burer2005}, in particular for the analysis of power systems \cite{lavaei2012zero,Marecek2017},
they are somehow too technical (cf. Theorem 4.1 in \cite{burer2005}). 
Based on these considerations, in this work we limit ourselves to the analysis of local convergence.

We first show that in the case of time-invariant input, under Assumption~\ref{asmPL}, the randomized coordinate-descent algorithm exhibits a linear rate of convergence \cite{karimi2016linear}.

\begin{theorem}[Extension of Theorem 6 in \cite{karimi2016linear}]
\label{th:lin_coo}
Let $\mu \in [0,\bar \mu]$ be fixed and
$\xi^*$, $r$, and $\sigma_\cL\leq dL$ be such that 
Assumption~\ref{asmPL} holds, and $\xi^0, \xi^1, \dots \in B_r(\xi^*)$.
Then the randomized coordinate-descent algorithm \eqref{eq:updateopt_compact}, with $i^k$ being chosen at each iteration uniformly at random from $\{1,\dots,d\}$, 
for solving \eqref{f+g}
enjoys the linear convergence rate
\begin{align}
\label{eq:local_linear_cvrg_rate}
\mathbb{E}[ \cL(\xi^k,\mu) - \cL^*] \leq \left( 1 - \frac{\sigma_\cL }{ d L}\right)^k[ \cL(\xi^0,\mu) - \cL^*],
\end{align}
where $L$ is as defined in Lemma~\ref{lemma-Phi} and $\cL^* : = \cL(\xi^*,\mu)$.
\end{theorem}

The proof directly follows from \cite{karimi2016linear}. 

We now bound the tracking error 
in the case of time-varying input, when $\cL$ changes over time.
In this case, we run a single iteration of our algorithm for each time step, 
before obtaining new inputs. 
In the following, $\cL^k(\xi,\mu)$ denotes the time-varying $\cL$ at each time $k$,
and we assume that the changes of $\cL$ are uniformly bounded.

\begin{assumption}\label{as:varying}
The change of the value of function $\cL^k$ at two subsequent instants $k-1$  and $k$ of time is bounded from above:
$$
|\cL^{k}(\xi,\mu) - \cL^{k-1}(\xi,\mu)| \leq e, \quad\textrm{for all } \xi \in \cY
$$
for all instants $k$. 
\end{assumption}


We are now ready to measure the performance of the tracking using the randomized coordinate-descent algorithm in the time-varying case.

\begin{theorem}
\label{theorem.inexact}
Let $\mu \in [0,\bar \mu]$ be fixed and
$\xi^{*,k}$, $r$, and $\sigma_\cL$ be such that 
Assumption~\ref{asmPL} and the Lipschitz condition~\eqref{lip_coo} are satisfied, uniformly in time. Furthermore, let Assumption~\ref{as:varying} hold and $\xi^0, \xi^1, \dots \in B_r(\xi^*)$.
Then the randomized coordinate-descent algorithm \eqref{eq:updateopt_compact} for solving \eqref{f+g} with $\cL^k(\xi,\mu)$ instead of $\cL(\xi,\mu)$, where $i^k$ is chosen uniformly at random from $\{1,\dots,d\}$, converges with linear rate to an error upper-bound
\begin{multline}
\label{eq:tracking_error}
\mathbb{E}[ \cL^{k}(\xi^k,\mu) - \cL^{*,k}] \\
\leq \left( 1 - \frac{\sigma_{\cL}}{ d L}\right)^k[ \cL^{0}(\xi^0,\mu) - \cL^{*,0}] + \frac{2e\cdot dL}{ \sigma_{\cL}}.
\end{multline}
Furthermore, the tracking error is
\begin{align}\label{eq.asympt_error}
\limsup_{k \to \infty}\mathbb{E}[ \cL^{k}(\xi^k,\mu) - \cL^{*,k}] \leq \frac{2e\cdot dL}{\sigma_{\cL}}.
\end{align}
\end{theorem}

\begin{proof}
The result follows from inequality~\eqref{eq:local_linear_cvrg_rate}, together with the triangle inequality and the sum of a geometric series. Omitting for brevity the dependency on $\mu$, for each $k$ we have
\begin{multline}
\label{errorbound2use}
\mathbb{E}[\cL^{k-1}(\xi^k) - \cL^{*,k-1}]
\\
\leq \left(1 - \frac{\sigma_{\cL}}{dL}\right)[\cL^{k-1}(\xi^{k-1}) - \cL^{*,k-1}].
\end{multline}
By summing and subtracting $\mathbb{E}[\cL^{k}(\xi^k)-\cL^{*,k}]$ on the left-hand-side, after trivial manipulation we obtain
\begin{multline}
\mathbb{E}[\cL^{k}(\xi^k) - \cL^{*,k}] \leq \left(1 - \frac{\sigma_{\cL}}{dL}\right)[\cL^{k-1}(\xi^{k-1}) - \cL^{*,k-1}] +
\\
|\mathbb{E}[\cL^{*,k}- \cL^{*,k-1}]|+|\mathbb{E}[\cL^{k}(\xi^k) - \cL^{k-1}(\xi^{k})]|\leq
\\
\left(1 - \frac{\sigma_{\cL}}{dL}\right)^k[\cL^{0}(\xi^{0}) - \cL^{*,0}] +2e\sum_{j=0}^{k-1}\left(1 - \frac{\sigma_{\cL}}{dL}\right)^j
\end{multline}
The last two term can be bounded exploiting Assumption~\ref{as:varying}, and we get~\eqref{eq:tracking_error} from the summation of the partial geometric series: $\sum_{j=0}^{k-1}(1-c)^j=[1-(1-c)^k]/c\leq 1/c$. The asymptotic inequality~\eqref{eq.asympt_error} follows immediately.
\end{proof}
The bound given in~\eqref{eq.asympt_error} quantifies the largest expected discrepancy between the optimum $\cL^{*,k}$ and the approximate optimum
$\cL^k(\xi^k,\mu)$ at iteration (time) $k$, when $k$ tends to infinity. 
More precisely, as time flows, the on-line randomized coordinate-descent algorithm produces a sequence of approximately optimal costs that  asymptotically approach the optimal costs trajectory, converging linearly with a rate that is dependent on the properties of the objective cost function. Notice that the asymptotic bound depends also on the speed at which the problem changes during time. 
If we wanted to run more iterations at each time step $k$, the analysis would be the same as in the off-line case 
(Theorem~\ref{th:lin_coo}) and therefore a tracking error would not be available. 
However, this approach could be impossible in frameworks where inputs change faster than the computation required for a single algorithm iteration. 
 



\section{Per-Iteration Complexity}
\label{sec:periteration}
We now study the complexity of one iteration of the method that we proposed, analyzing the complexity of one epoch of the coordinate-descent algorithm, during which the iterations go over each coordinate $i$ in increasing order.

\begin{lemma}
\label{noFlops}
Let $p$ denote the maximal number of non-zero elements of 
a row of $y$ (in other words, $p$ is the maximal number of nodes that any node can be connected to).
Sequentially visiting each coordinate $i$, the coordinate-descent algorithm
requires $(32p+102) N^2 + (32p+116)N_{\cG}N  - 2N + (16p+92)N_{\cG}$ flops
and $6(N+N_{\cG})$ roots evaluations for a uni-variate cubic polynomial.
The update of a single coordinate requires at most
$16(N+N_{\cG})p+58N_{\cG}+51N-8$ flops and $6N$ root evaluations for a uni-variate cubic polynomial. 
\end{lemma}

\begin{proof}
We first recall that the evaluation of the traces of high-dimensional quadratic forms can exploit sparsity.
Consider for instance
matrix $Y_i$ defined in~\eqref{defYk}, where $y_i = e_i e_i^T y$ and $y$ is the admittance matrix of the system. 
Evaluating the trace of the quadratic form $\tr(Y_i xx^T)$ requires at most 
$8p$ flops, where for realistic power systems $p$ is constant and $p\ll n$. 
Moreover, terms in which $M_i$ appears \eqref{defMk} can be simplified, for example with
\[
\tr(M_i xx^T) = x_i^2 + x_{i+N}^2,
\]
and they can hence be evaluated in 3 float-float operations, and in one flop if $x_i$ or $x_{i+N}$ is a variable.

We proceed recalling that we are minimizing coordinate-wise. 
Enumerating the local minima of $ax^4 + bx^3 + cx^2 +dx +e$ is equivalent to solve the cubic equation 
$x^3 + 3b/4ax^2 + c/2ax + d/4a = 0$.
Thus, we have unconstrained optimization problems for $x$ and box-constrained quartic optimizaton problems for $t$ and $g$, both taking a similar solving cost.

We now sum up the numbers: the
evaluation of one coordinate of $\tr(Y_i xx^T)$
costs $8p$ flops. 
It then takes $11$ additional operations to compute the coefficients 
for $[\tr(Y_i xx^T)+P_{l.i}]^2$, and $3$ operations when we have $M_i$ instead of $Y_i$. 
Assuming that the number of generators is $N_{\cG}$ it is possible to check by plain counting that we need in total
$16(N+N_{\cG})p+58N_{\cG}+51N-8$ flops for coefficient evaluations coming from (\ref{lifted}b-c-d-f). 
Since each epoch performs
$2N$ such coordinate-wise iterations, it has a total cost of $(32p+102)N^2 +(32p+116)N_{\cG}N-16N$ flops plus $6N$ root evaluations for a uni-variate cubic polynomial.

Similarly, for $t_i, i\in\mathcal{G}$ \eqref{deft}, the evaluations of the coefficient are necessary only for the quadratic and quartic terms, where the quadratic terms 
$\left[(t_i+P_{l,i})^2  + (g_i+Q_{l,i})^2 - z_{i}\right]$ 
and 
$\left[\tr(Y_ixx^T)-t_i\right]^2$
respectively take $6$ and 
$(8p+2)$ flops. 
The quartic term takes $11$ more operations. 
The per-epoch update of $t_i$ comes at the cost of 
$(8p+38)N_{\cG}$ 
flops plus $3N_{\cG}$ root evaluations. 
Updates in $g$~\eqref{defg} have the same cost.

Finally, for $h_i, i\in\mathcal{N}$ and $z_i, i\in\mathcal{G}$, we have box-constrained quadratic optimization problems, 
and it is not hard to calculate that the coefficients evaluation requires respectively $12$ and $14$ flops, and solving 
a quadratic problem takes only $2$ flops. Hence, the cost per-epoch is of $14N$ and $16N_{\cG}$ flops  
for $h_i$ and $z_i$ respectively.
\end{proof}

Summarizing, the total cost of one epoch is $(32p+102) N^2 + (32p+116)N_{\cG} N  - 2N + (16p+92)N_{\cG}$ flops plus $6(N+N_{\cG})$ roots evaluations a cubic polynomial. 
Giving a bound on the number of flops necessary for the roots evaluations of a cubic polynomial can be hard, as the computations involve square and cubic roots of scalars. 
In a computation model where taking the root of a scalar requires 1 flop,
such as in the BSS machine \cite{blum1989theory},
finding the roots of a cubic polynomial requires 31 flops.
The update of a single coordinate in such a model hence needs $16(N+N_{\cG})p+58N_{\cG}+144N-8$ flops at most.



This allows to bound the expected tracking error $\cL^{k}(\xi^k,\mu) - \cL^{*,k}$ by quantities
which are easier to evaluate. 
Let us now consider the number of flops that are required in order to guarantee a given accuracy in terms of the error bound between two updates of the inputs.

 \begin{theorem}
\label{V2}
Let Assumptions~\ref{asmPL} and \ref{as:varying} hold. 
Let $p$ be defined as in Lemma \ref{noFlops}, $\sigma_l := [ \cL^{0}(\xi^0,\mu) - \cL^{*,0}]$,
and a parameter $\sigma_p := dL/\sigma_{\cL}$.
The number of flops that a BSS machine needs to perform between two consequent updates of the inputs in order to ensure that the error is bounded by $E_k := \mathbb{E}[ \cL^{k}(\xi^k,\mu) - \cL^{*,k}]$ is
\begin{align}
    \label{flopserror}
[16(N+N_{\cG})p+58N_{\cG}+144N-8]\frac{\log (E_k - 2e\cdot\sigma_p )}{\log \sigma_l}.
\end{align}
\end{theorem}

\begin{proof}
The linear convergence established in Theorems~\ref{th:lin_coo} and \ref{theorem.inexact} means 
$E_k$ is bounded by a function of $\sigma_p$ raised to the $k$-th power.
In turn, $k$ is bounded from above by the ratio of the total number of flops between two updates and a
worst-case bound on the numbers of flops required for one coordinate-wise update,
which by Lemma~\ref{noFlops} is $16(N+N_{\cG})p+58N_{\cG}+144N-8$, i.e., $O(Np)$.
We conclude by substituting $\sigma_p, \sigma_l$ into \eqref{errorbound2use}, solving for $\sigma_p^k$, substituting the ratio instead of $k$,
and taking the logarithm of both sides.
\end{proof}

Since modern computers are indeed not BSS machines, and their behavior is quite complex, 
the bound \eqref{flopserror} may not be a perfect estimate of the actual run time,
but it does provide some guidance as to the requirements on computing resources.
Specifically, the run-time to a constant error bound grows with $O(Np)$, 
when $\sigma_l$ and $\sigma_p$ are constant. 



\begin{figure}[t] 
\centering
\vspace{.25cm}
\includegraphics[width=.40\textwidth]{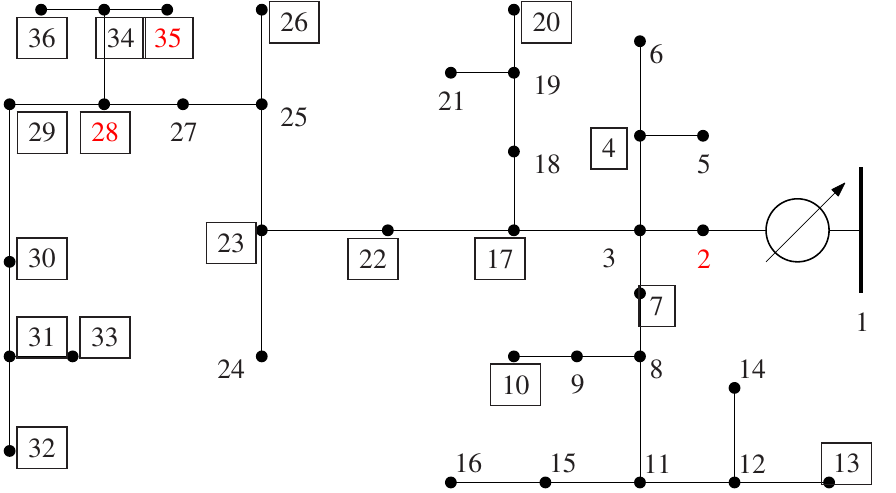}
\caption{IEEE 37-node feeder, as amended by Dall'Anese and Simonetto \cite{7480375}: 18 photovoltaic systems (secondary transformers) are marked with a box.}
\label{F_feeder}
\end{figure}

\section{Empirical Results}

To validate our approach, we consider the application of a distribution network with high penetration of photovoltaic systems, 
introduced by Dall'Anese and Simonetto \cite{7480375} (we remark that, indeed, the application of our method is not restricted to radial networks).
The modified network is a single-phase variant of the IEEE 37-node test feeder, obtained by replacing loads of 18 secondary transformers (see the boxed nodes in Figure~\ref{F_feeder}) 
with real load data from Anatolia, California, sampled with $1$ Hz frequency in August 2012~\cite{Bank13}. 
Furthermore, the generation at photovoltaic plants is simulated based on real solar irradiance data in \cite{Bank13},
with rating of these inverters at $300$ kVA at node $3$; $350$ kVA at nodes $15, 16$, and $200$ kVA for all other inverters.
We set the voltage minimum $V_{\mathrm{min}}$ and the maximum $V_{\mathrm{max}}$ to be 0.95 pu and 1.05 pu, respectively. 
The goal is to provide insights that how reliable different controllers are to avoid overvoltages and keep stability during different periods of the day.
The solar irradiance data also have the granularity of one second. 
Other parameters are kept intact.



We evaluate the performance of the modified network at 3 Hz frequency instead of the 1 Hz update, and present the numerical results in Figures~\ref{fig:midnight} and \ref{fig:fivepm}
The top rows present the voltage profile for nodes 2, 15, 28, and 35, where the performance improves when compared to Figure~\ref{F_feeder118} by Dall'Anese and Simonetto \cite{7480375}. Even during the solar peak hours between 10:00 and 14:00, the voltage regulation is well enforced. 
Taking a closer look in the zoomed-in Figure~\ref{fig:fivepm}, there is little volatility in the voltage.
The middle plots report the cost achieved \,--\, $\sum_{i \in \cG} c_q (Q_i^k)^2 + c_p (P_{\textrm{av},i}^k)^2$.
In the bottom plots, we present a measure of infeasibility defined as
\begin{align}
\label{infeasibility}
T&(x, t, g, h, z) :=
 \sum_{i\in\mathcal{N}} \left[\tr(Y_ixx^T)+\omega_i^T x-t_i\right]^2 \nonumber \\
&  + \sum_{i\in\mathcal{N}} \left[\tr(\bar Y_i xx^T)+ \bar\omega_i^T x-g_i\right]^2 
+\sum_{i\in\mathcal{N}} \left[\tr(M_i xx^T)-h_i \right]^2 
\nonumber \\ 
& + \sum_{i\in\mathcal{G}} \left[(t_i+P_{l,i})^2  + (g_i+Q_{l,i})^2 - z_{i}\right]^2, 
\end{align} 
which originates from the constraints (\ref{lifted1}-\ref{lifted10}) without considering the bound constraints, and compare it with the linearization of Dall'Anese and Simonetto~\cite{7480375}. 
To be consistent with the authors, we use the same parameters and set $\nu = 10^{-3}$, $\epsilon = 10^{-4}$, $\alpha = 0.2$,  $c_p = 3$,  $c_q = 1$, 
$\bar{f}^k(\bu^k) = \sum_{i \in \cG} c_q (Q_i^k)^2 + c_p (P_{\textrm{av},i}^k - P_i^k)^2$.
To clearly demonstrate the efficiency, besides the full measure of infeasibility $T$ in \eqref{infeasibility}, we also evaluate $T'$, a lower bound of the infeasibility \eqref{infeasibility}, where we only consider infeasibility of active generators and voltage regulations by ignoring the terms $\sum_{i\in\mathcal{N}\backslash\mathcal{G}} \left[\tr(Y_i xx^T)+ \omega_i^T x-t_i\right]^2 
+ \sum_{i\in\mathcal{N}\backslash\mathcal{G}} \left[\tr(\bar Y_i xx^T)+ \bar\omega_i^T x-g_i\right]^2$. These terms correspond to the non-generator constraints, that is, \ref{lifted9} and \ref{lifted10} in the lifted formulation \eqref{lifted} and \eqref{mg-balance-L} in the original formulation \eqref{Pmg} of \cite{7480375},
which are most affected by the linearization.
By comparing the three forms of infeasibility on the same logarithmic axis, infeasibility $T$ of our approach is approximately 4 orders of magnitude better than 
the lower bound $T'$ on the infeasibility of the linearization,
and more than 8 orders of magnitude better than the infeasibility $T$ of the linearization. Moreover, $T'$ of linearization projects obvious spikes during 10:00 -- 14:00 indicating weak infeasibility enforcement while PV generation exceeds the demand.

 \begin{figure}[tb]
 \center
\includegraphics[scale=0.42]{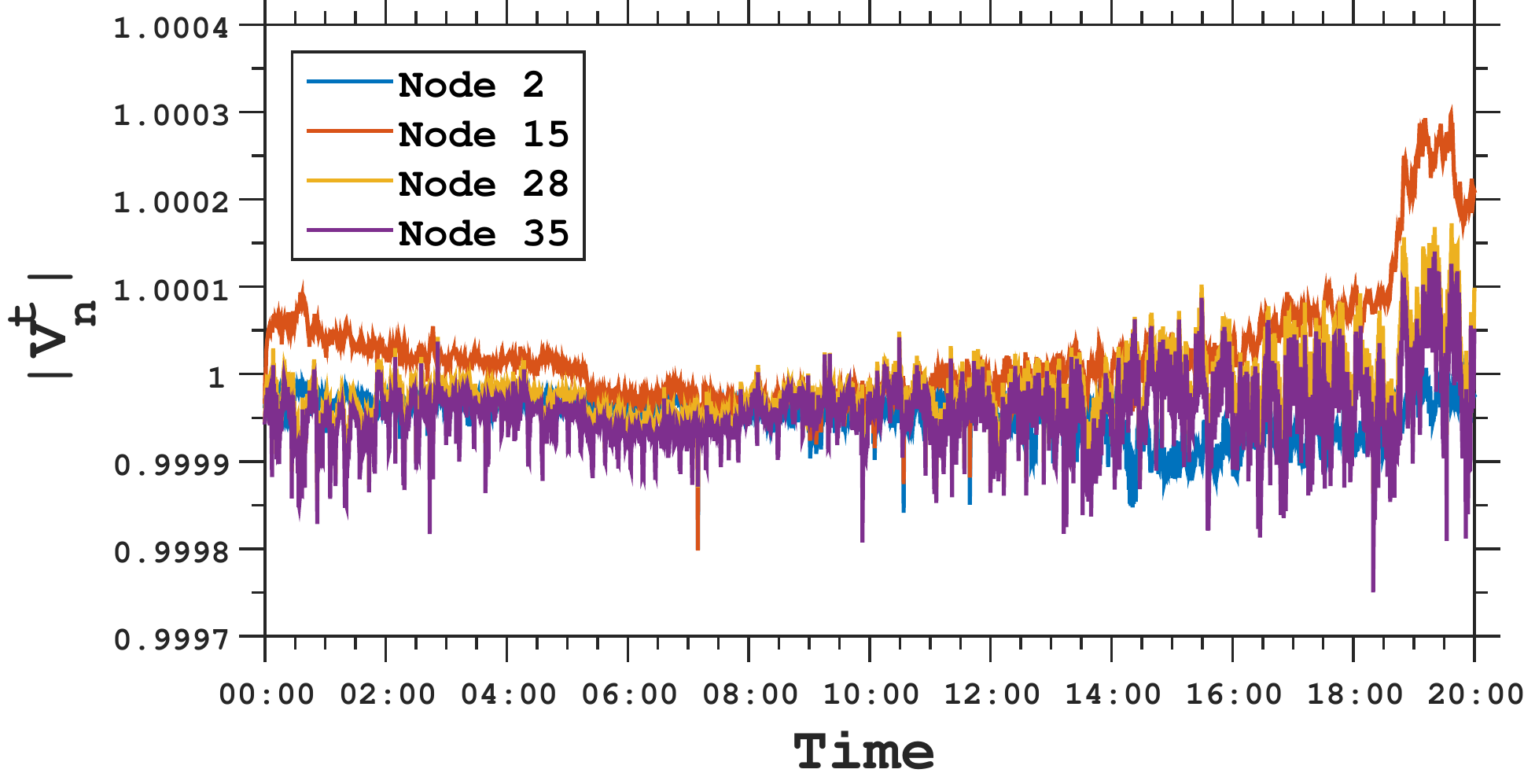}
\\ $ $\\
\vspace{-1em}
\qquad \includegraphics[scale=0.4]{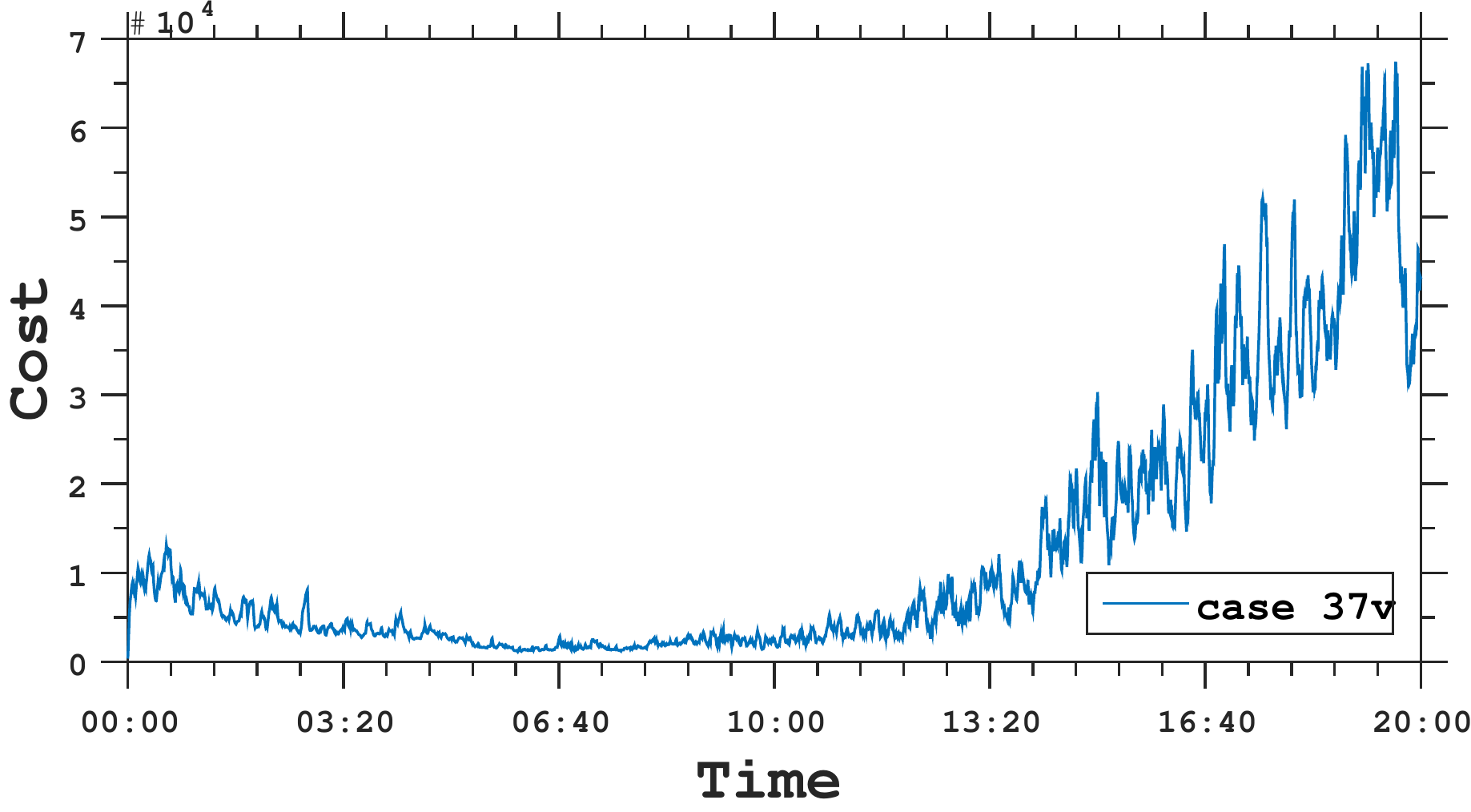}
\\ $ $\\
\vspace{-1em}
\quad \includegraphics[scale=0.4]{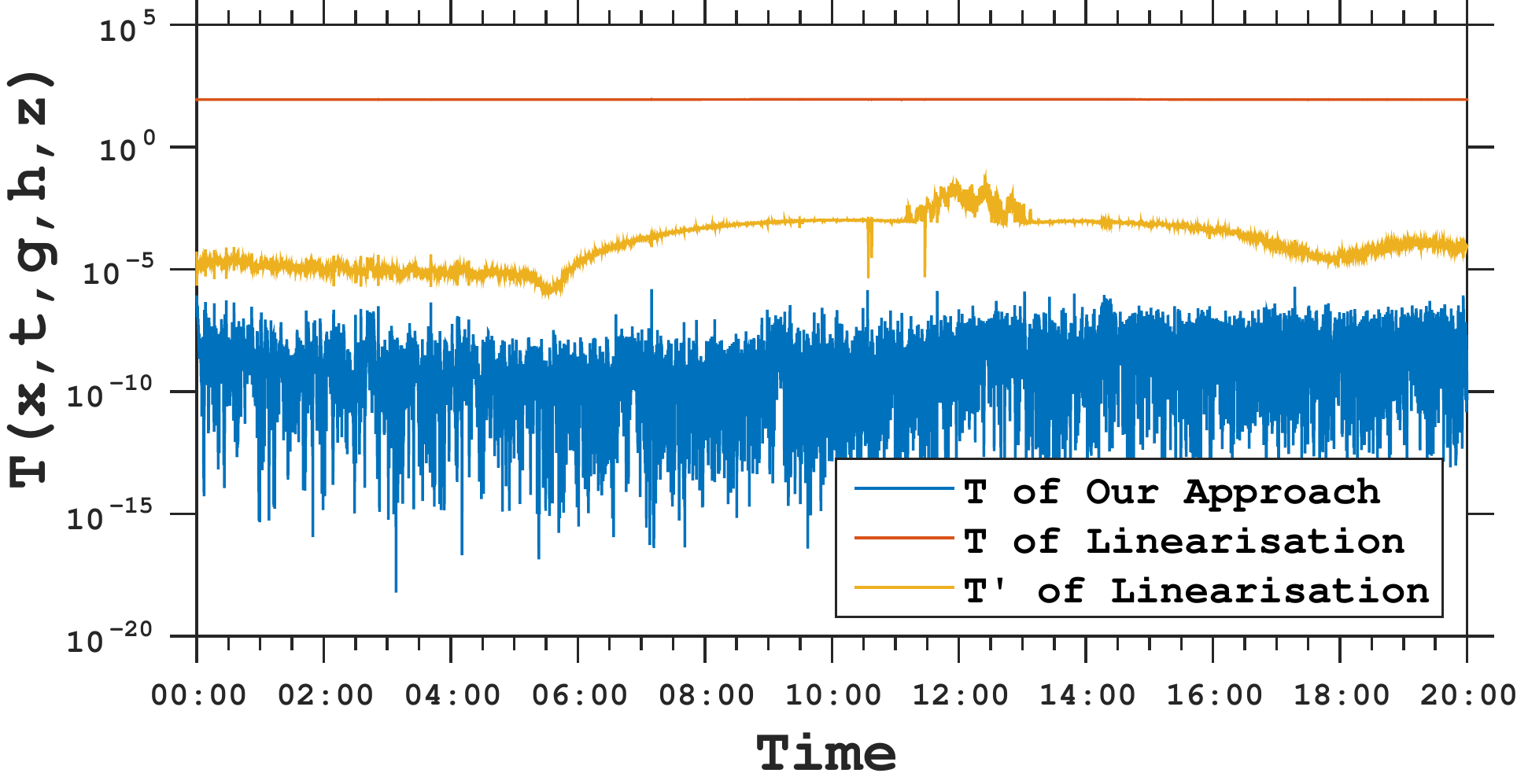}
\caption{The performance on the feeder of Figure~\ref{F_feeder}, from midnight till 8pm.}
 \label{fig:midnight}
\end{figure}

 \begin{figure}[tb]
 \center
\includegraphics[scale=0.42]{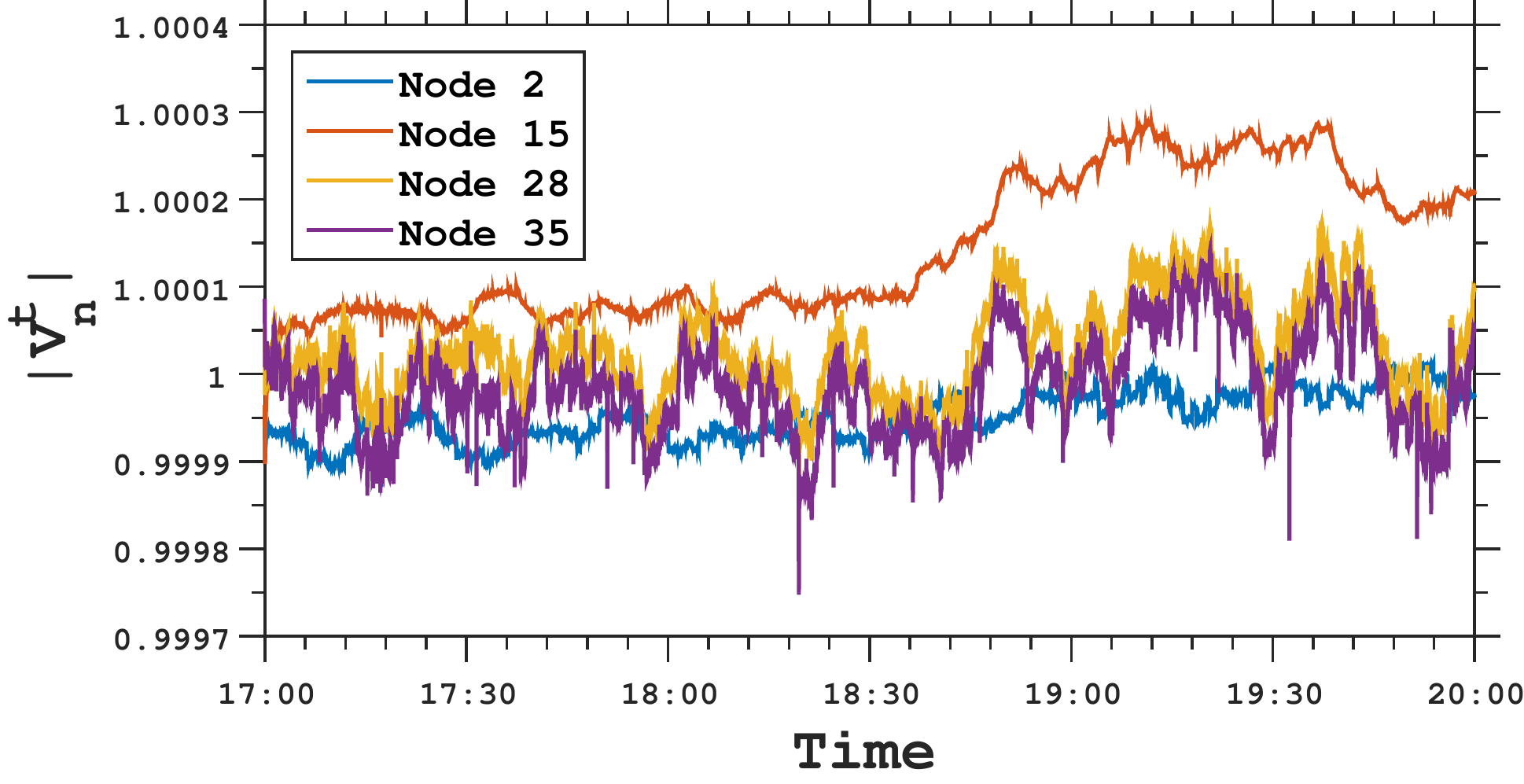}
\\ $ $\\
\vspace{-1em}
\qquad \includegraphics[scale=0.4]{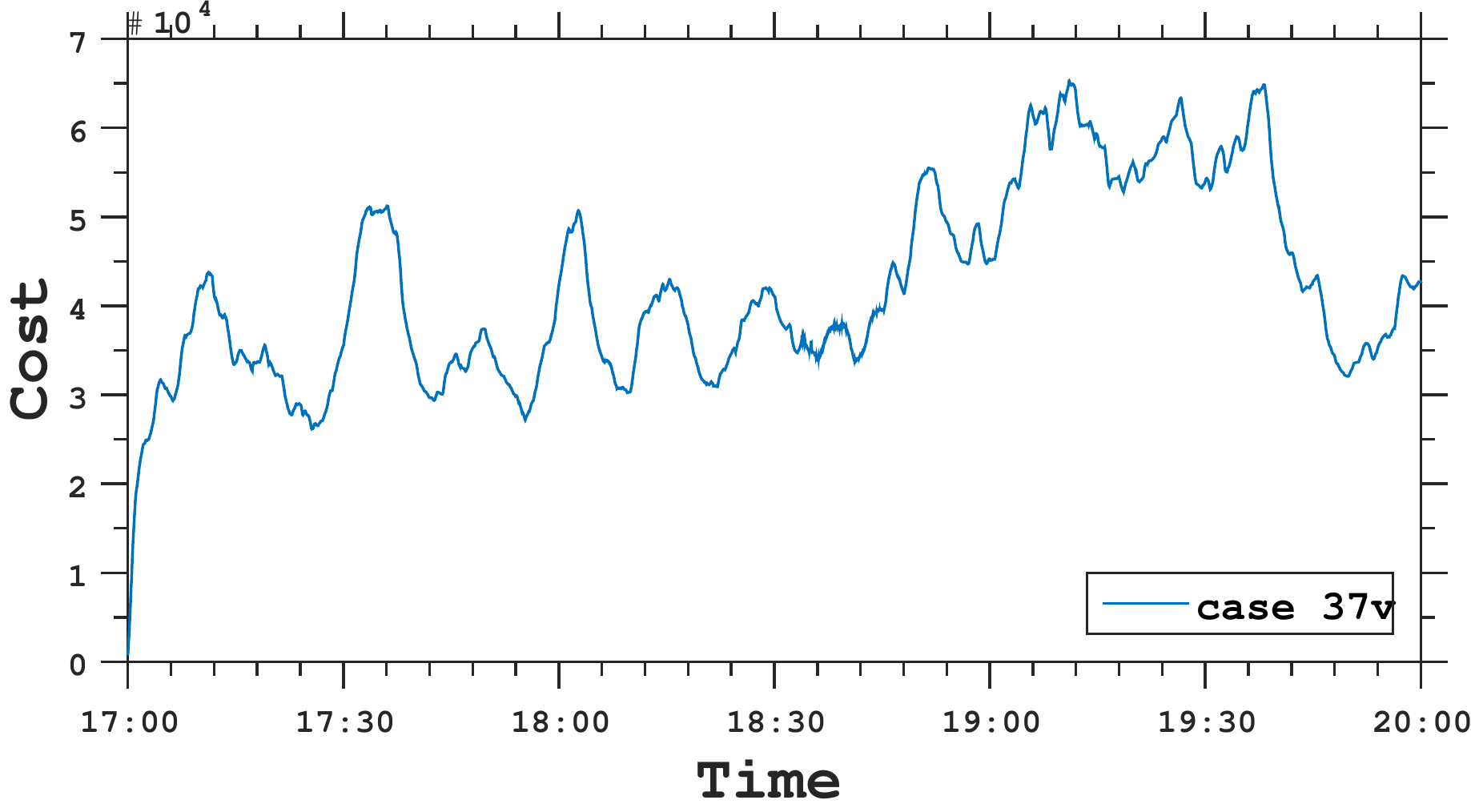}
\\ $ $\\
\vspace{-1em}
\quad \includegraphics[scale=0.4]{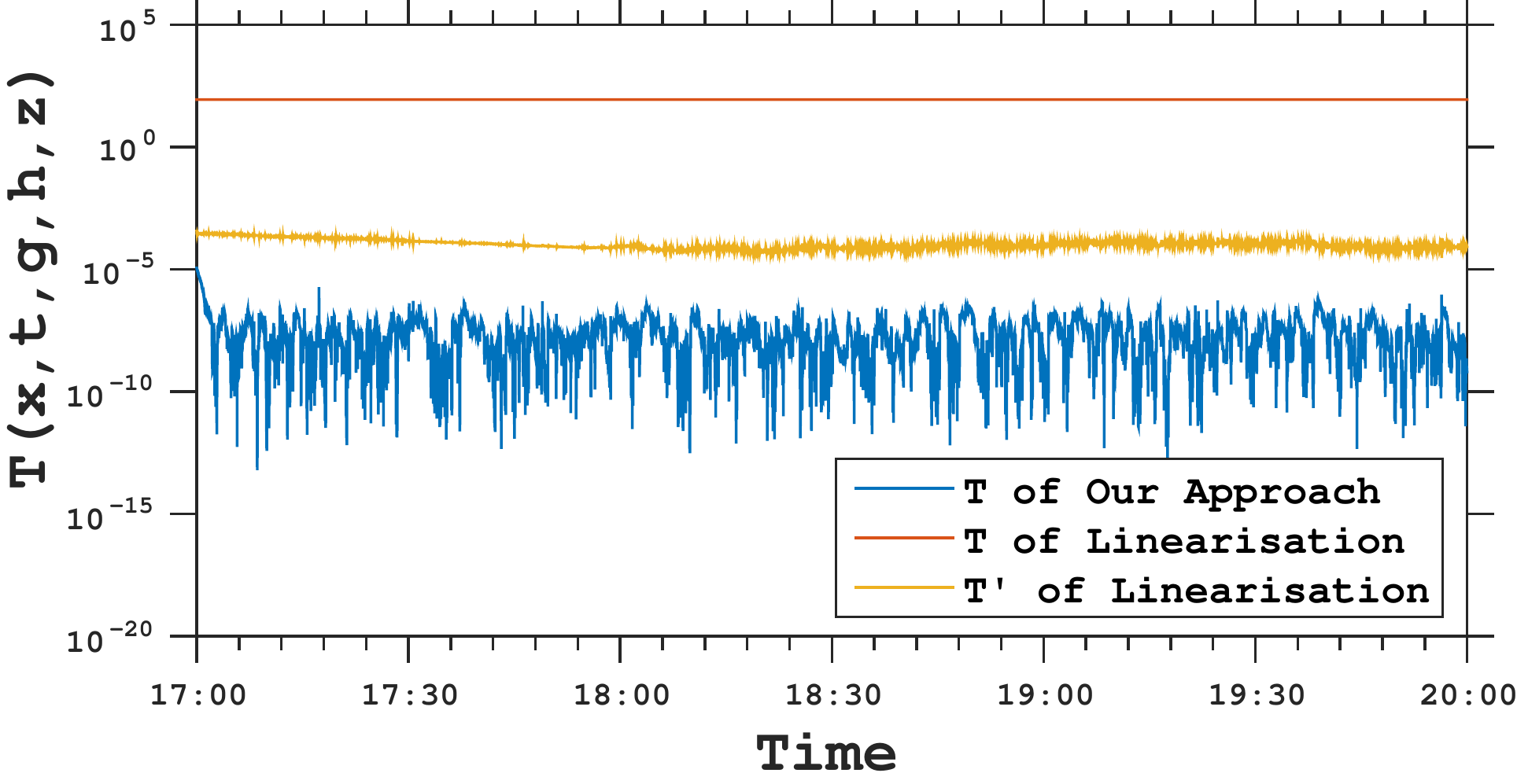}
\caption{A zoom in on the performance on the feeder of Figure~\ref{F_feeder}, from 5pm till 8pm.}
 \label{fig:fivepm}
\end{figure}


In the second experiment, we consider the IEEE 118-bus test system and employ time-varying loads and maximum active power generation (see Figure~\ref{F_feeder118}), similar as what has been done on the 37-node feeder. All generators come with time-varying maximum active power. 60 nodes, composed of both generators and non-generators, and admit time-varying loads. The voltage limits $V_{\mathrm{max}}$ and $V_{\mathrm{min}}$ are set to $1.06$ pu and $0.94$ pu, respectively. In the middle plots, we present the cost $\sum_{i \in \cG} c_2 (P_{\textrm{av},i}^k)^2 + c_2 (P_{\textrm{av},i}^k)^2 + c_0$. In the bottom plots, we present the measure of infeasibility $T(x,t,g,h,z)$. All time-varying data are sampled at 1 Hz and we present the performance of the algorithm run with different frequencies, namely, 1 Hz in blue, 1/10 Hz in red, and 1/60 Hz in yellow (middle and bottom plots). In the top plots, we also provide the voltage profile for nodes $2, 12, 55, 70$, and $94$, where $12, 55, 70$ are generators.
In Figures~\ref{fig:20h} and~\ref{fig:6h}, the bottom plots reveal that the more time the solver is allotted, the lower infeasibility $T$ it can provide, even though this does not always entail better objective-function values, as shown in the middle plots.

\begin{figure}[t] 
\centering
\vspace{.25cm}
\includegraphics[width=.50\textwidth]{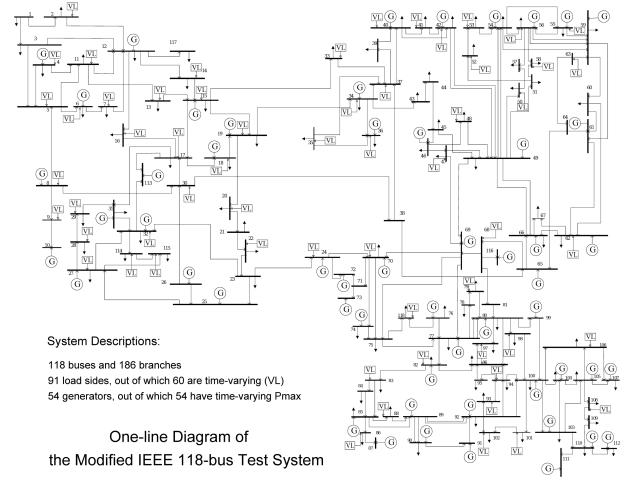}
\caption{IEEE 118-bus test system feeder, 54 generators and 60 buses with time-varying loads are marked with circles and boxes, respectively.}
\label{F_feeder118}
\end{figure}

 \begin{figure}[tb]
 \center
\includegraphics[scale=0.42]{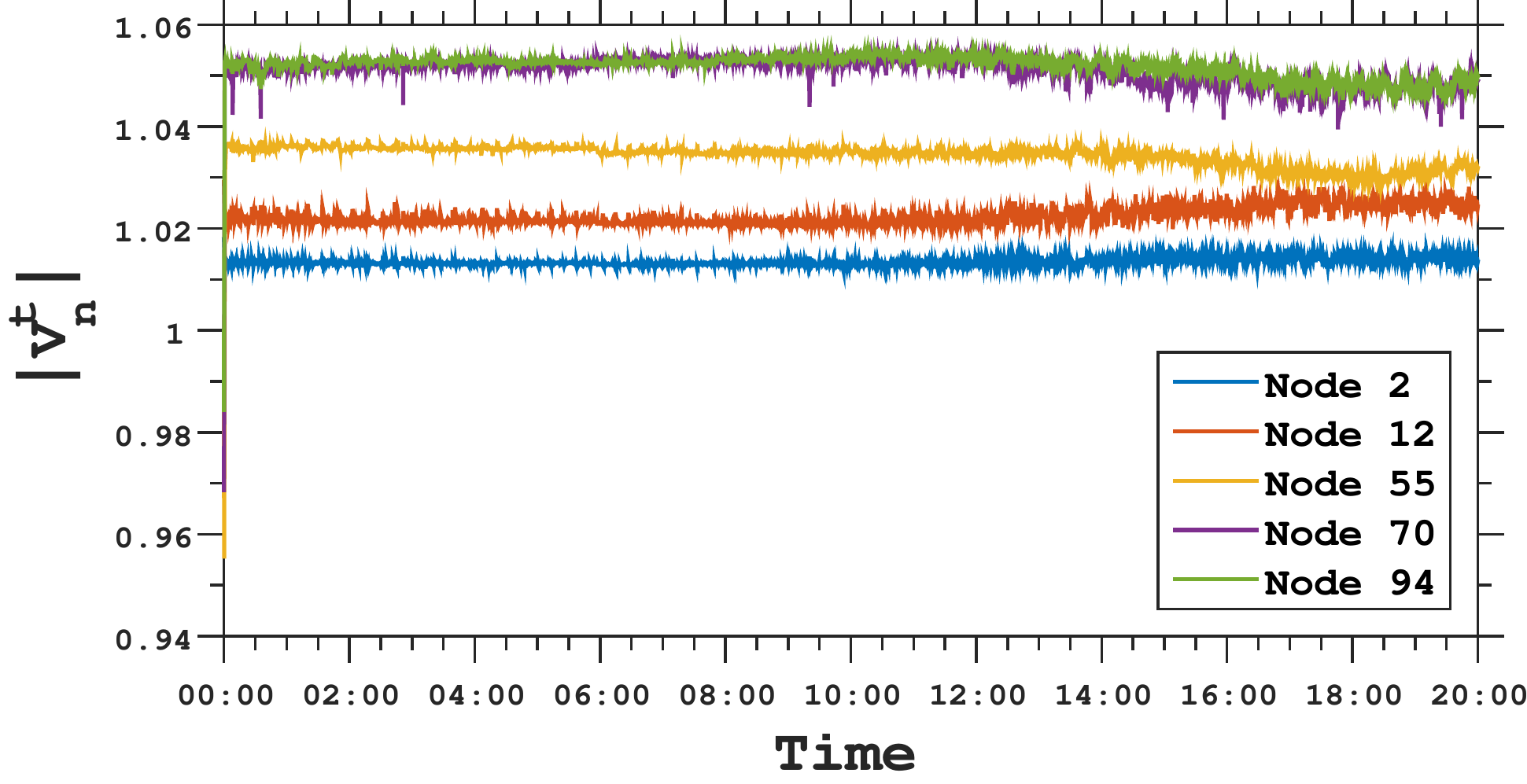}
\\ $ $\\
\vspace{-1em}
\ \includegraphics[scale=0.4]{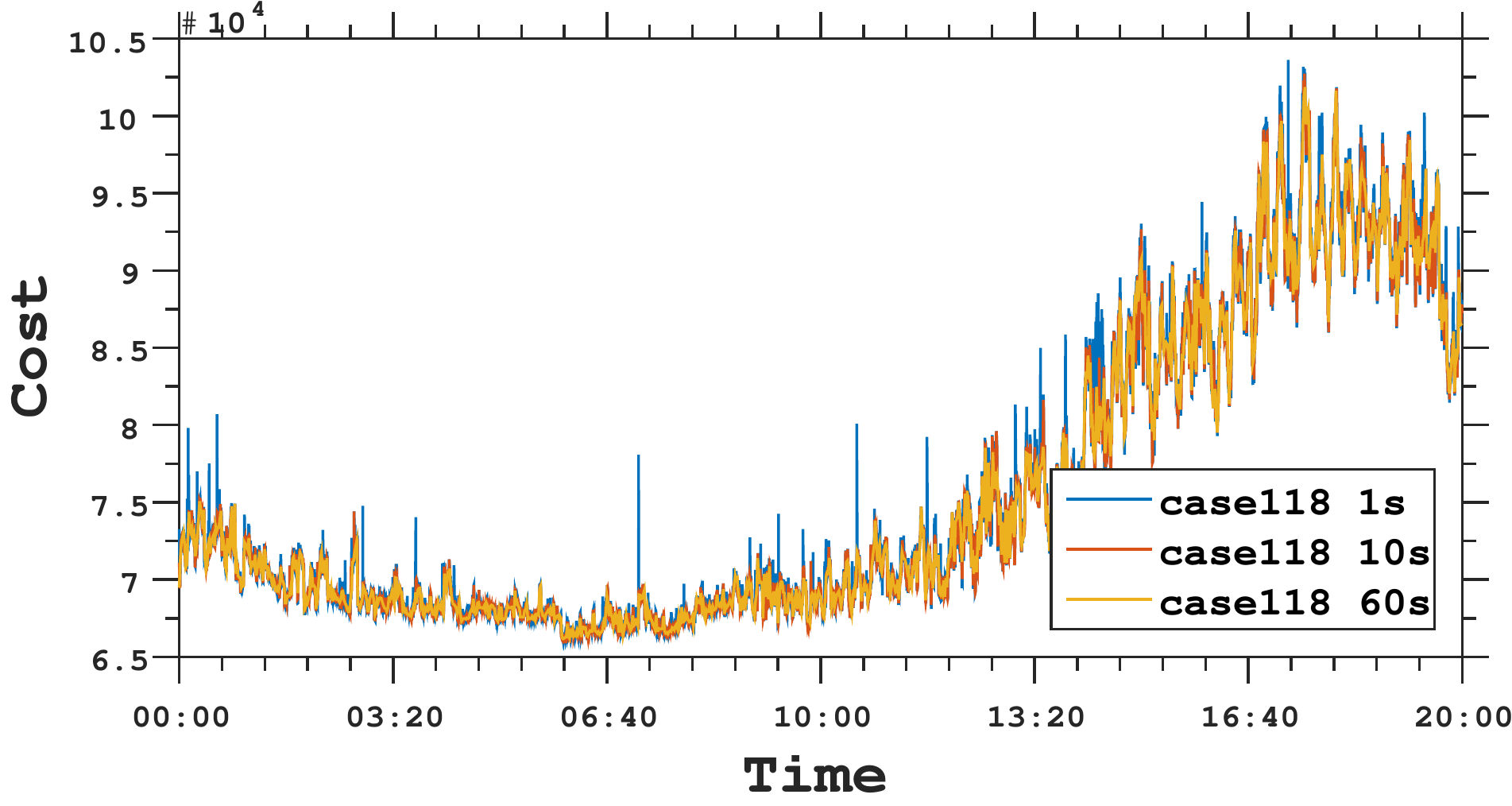}
\\ $ $\\
\vspace{-1em}
\ \includegraphics[scale=0.4]{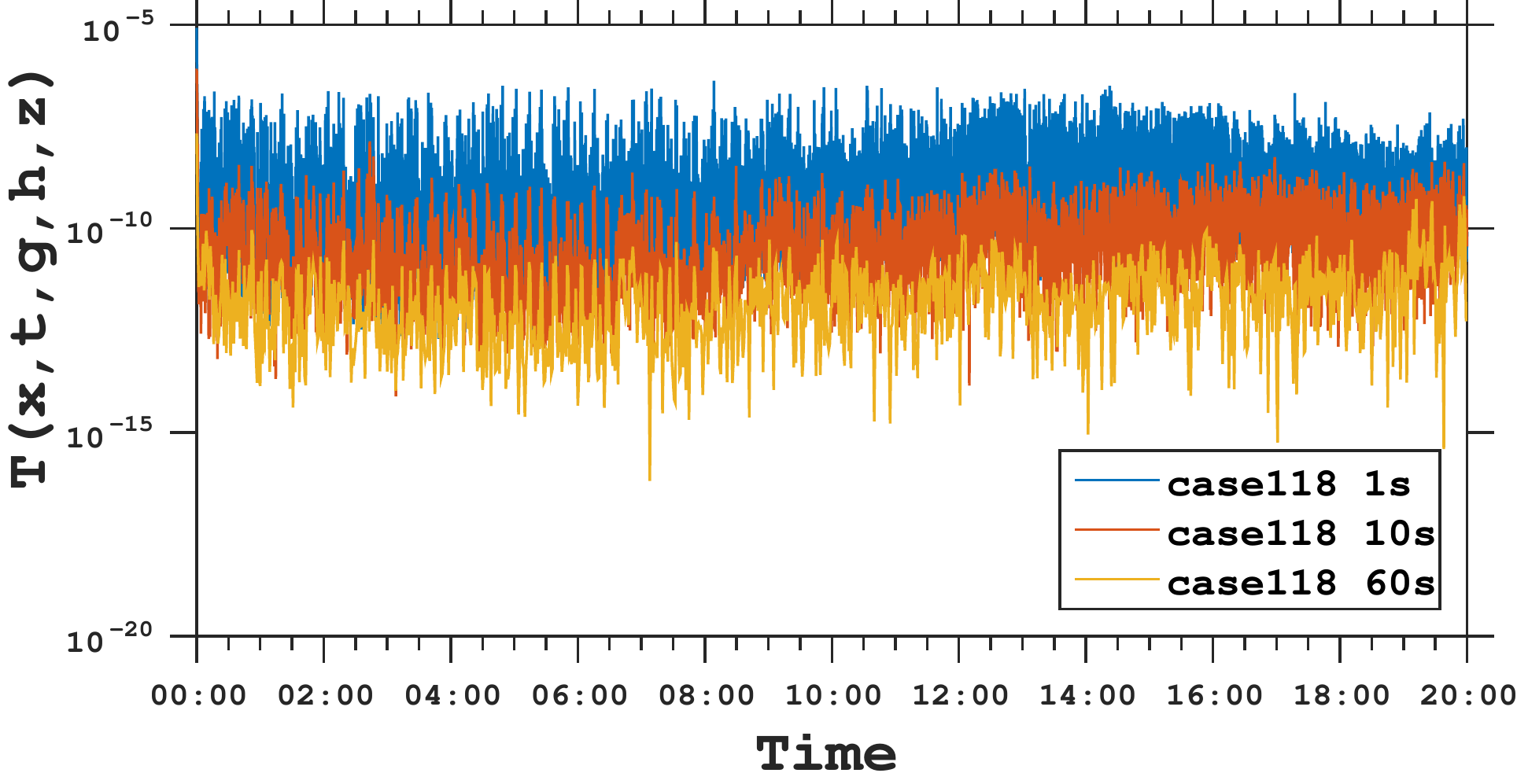}
\caption{The performance on the feeder of Figure~\ref{F_feeder118}, from midnight till 8pm.}
 \label{fig:20h}
\end{figure}

 \begin{figure}[tb]
 \center
\includegraphics[scale=0.42]{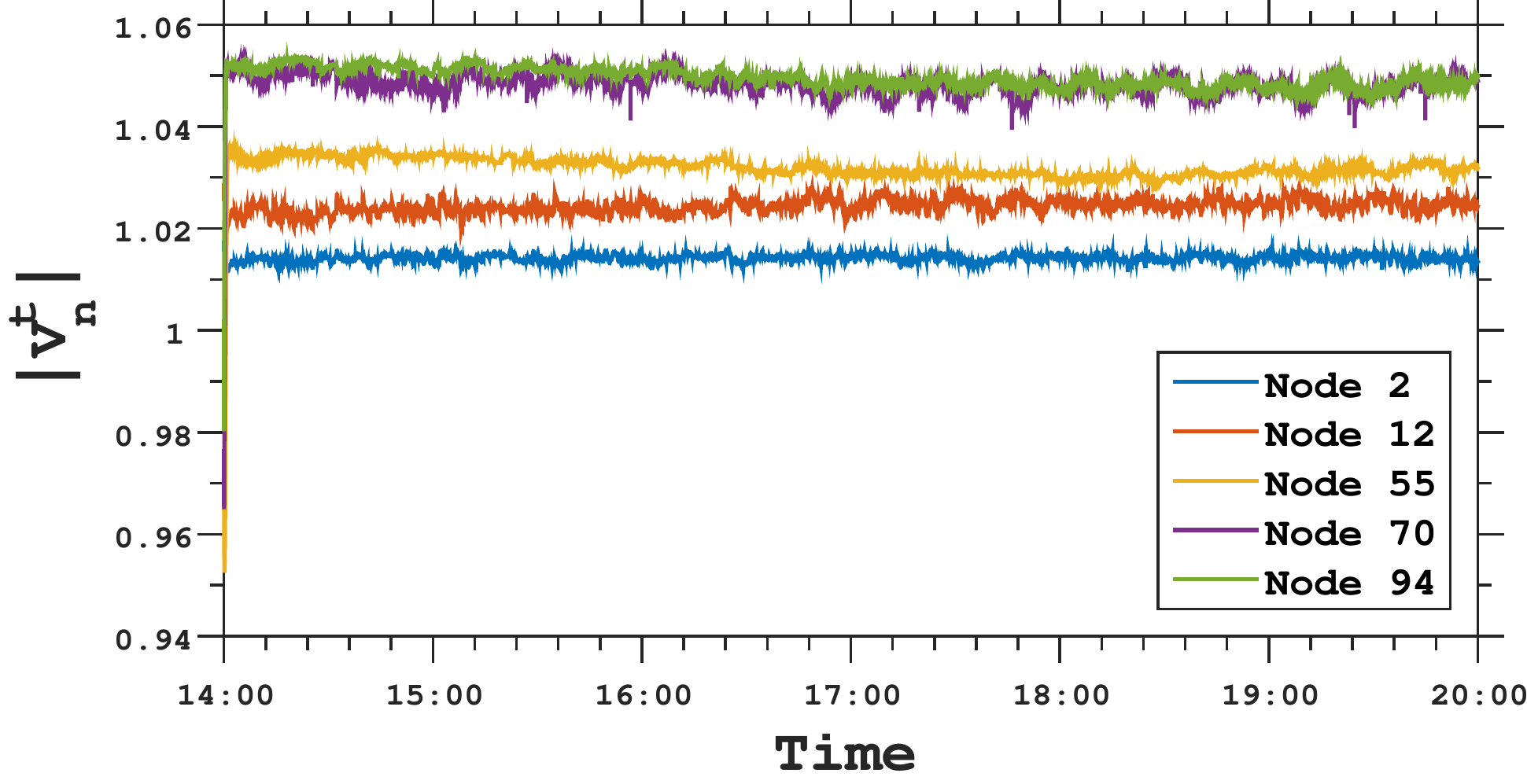}
\\ $ $\\
\vspace{-1em}
\ \includegraphics[scale=0.4]{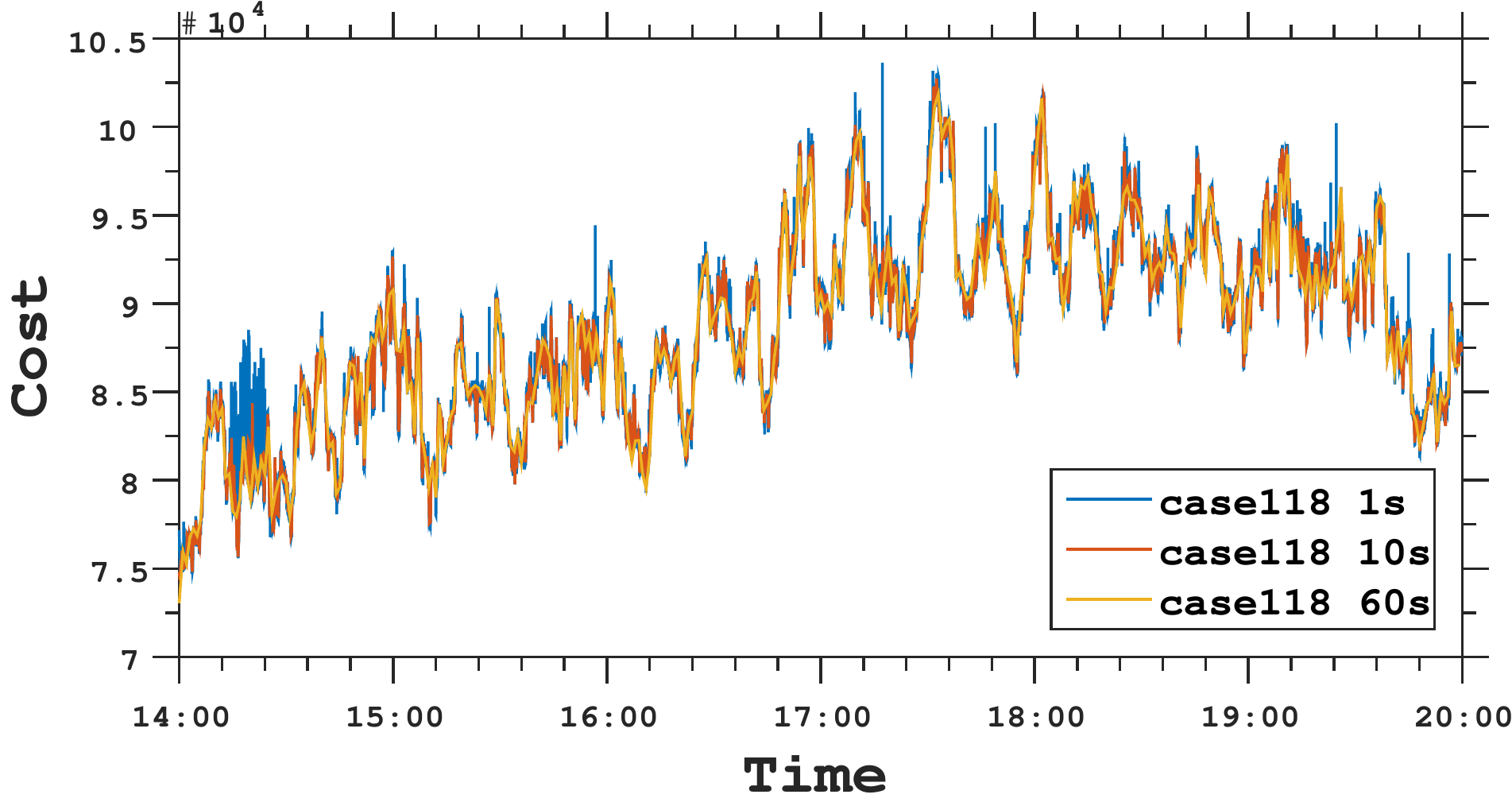}
\\ $ $\\
\vspace{-1em}
\ \includegraphics[scale=0.4]{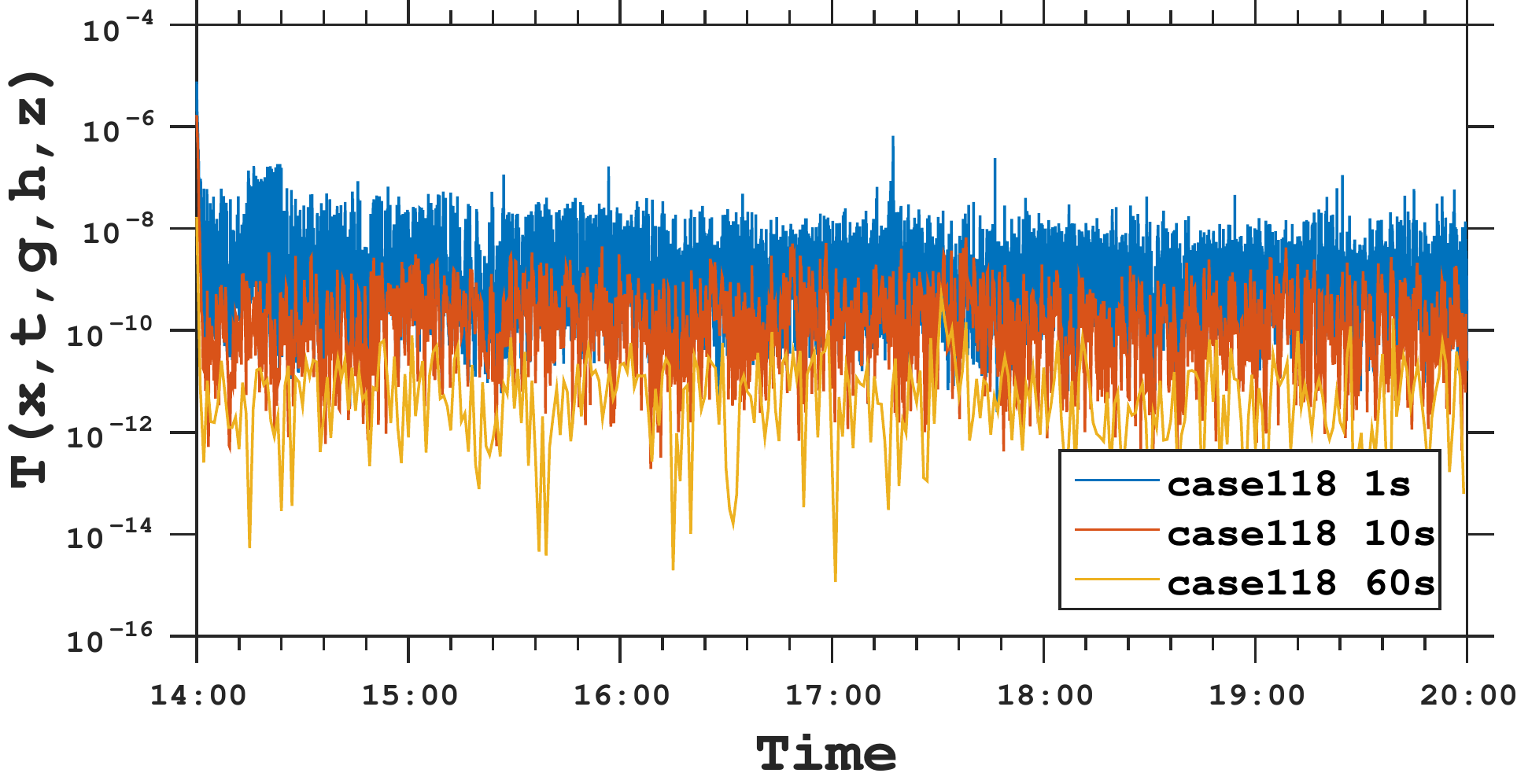}
\caption{A zoom in on the performance on the feeder of Figure~\ref{F_feeder118}, from 2pm till 8pm.}
 \label{fig:6h}
\end{figure}

\section{Conclusions}

Increasing volatility in optimal power flow parameters considerably increases the interest in seeking solutions to optimal power flows in the near real-time alternating current model.
Coordinate-descent algorithms seem well-suited to tracking solutions of optimal power flows.
Theoretically, they make it possible to analyze the number of flops per second a machine
should be capable of, in order to achieve a certain guarantee on the tracking error while dealing with
a power system of known dimension and loads and limitations of generation of known volatility.
In our analysis, we use a variant of the Polyak\,--\,{\L}ojasiewicz condition \cite{Polyak63,lojasiewicz1963propriete} 
Due to the appeal of allowing for non-convexity and non-unique optima, we imagine that there may be many subsequent applications.
As shown by computational experiments, due to the essentially linear per-iteration run-time the proposed algorithm performs very well.

\bibliographystyle{IEEEtran} 
\bibliography{pursuit,literature,jie,acopf}

\end{document}